\documentclass[11pt,reqno]{amsart}

\usepackage[a4paper,left=35mm,right=35mm,top=30mm,bottom=30mm,marginpar=25mm]{geometry}

\usepackage{amsfonts}
\usepackage{amsmath}
\usepackage{amssymb}
\usepackage{amsthm}
\usepackage[latin1]{inputenc}
\usepackage{eurosym}
\usepackage[dvips]{graphics}
\usepackage{graphicx}
\usepackage{epsfig}
\usepackage{hyperref}
\usepackage{dsfont}
\usepackage{color}

\usepackage[displaymath,mathlines]{lineno}

\allowdisplaybreaks

\usepackage{hyperref}

\usepackage{mathtools}
\mathtoolsset{showonlyrefs}

\usepackage{ifthen}
\makeindex

\renewcommand{\div}{\operatorname{div}}

\newcommand{\Rr}{{\mathbb{R}}}

\newcommand{\Tt}{{\mathbb{T}}}

\newcommand{\Hh}{{\overline{H}}}

\newcommand{\Uu}{{\mathcal{U}}}
\newcommand{\Vv}{{\mathcal{V}}}
\newcommand{\Ww}{{\mathcal{W}}}

\newcommand{\epsi}{\varepsilon}

\def\leq{\leqslant}
\def\geq{\geqslant}

\numberwithin{equation}{section}

\newtheoremstyle{thmlemcorr}{10pt}{10pt}{\itshape}{}{\bfseries}{.}{10pt}{{\thmname{#1}\thmnumber{
                        #2}\thmnote{ (#3)}}}
\newtheoremstyle{thmlemcorr*}{10pt}{10pt}{\itshape}{}{\bfseries}{.}\newline{{\thmname{#1}\thmnumber{
                        #2}\thmnote{ (#3)}}}
\newtheoremstyle{defi}{10pt}{10pt}{\itshape}{}{\bfseries}{.}{10pt}{{\thmname{#1}\thmnumber{
                        #2}\thmnote{ (#3)}}}
\newtheoremstyle{remexample}{10pt}{10pt}{}{}{\bfseries}{.}{10pt}{{\thmname{#1}\thmnumber{
                        #2}\thmnote{ (#3)}}}
\newtheoremstyle{ass}{10pt}{10pt}{}{}{\bfseries}{.}{10pt}{{\thmname{#1}\thmnumber{
                        A#2}\thmnote{ (#3)}}}

\theoremstyle{thmlemcorr}
\newtheorem{theorem}{Theorem}
\numberwithin{theorem}{section}
\newtheorem{lemma}[theorem]{Lemma}
\newtheorem{corollary}[theorem]{Corollary}
\newtheorem{proposition}[theorem]{Proposition}

\theoremstyle{thmlemcorr*}
\newtheorem{theorem*}{Theorem}
\newtheorem{lemma*}[theorem]{Lemma}
\newtheorem{corollary*}[theorem]{Corollary}
\newtheorem{proposition*}[theorem]{Proposition}
\newtheorem{problem*}[theorem]{Problem}
\newtheorem{conjecture*}[theorem]{Conjecture}

\theoremstyle{defi}

\newtheorem{hyp}{Assumption}

\theoremstyle{remexample}
\newtheorem{remark}[theorem]{Remark}

\DeclareMathOperator{\diverg}{div}

\theoremstyle{ass}

\begin{document}

\title[The variational structure of the mean-field games systems]{The variational structure and time-periodic solutions for mean-field games systems}

\author{Marco Cirant}
\address[M. Cirant]{
Dipartimento di Matematica ``Tullio Levi-Civita'', Universit\`a di Padova,
via Trieste 63, 35121 Padova (Italy)}
\email{cirant@math.unipd.it}

\author{Levon Nurbekyan}
\address[L. Nurbekyan]{
	King Abdullah University of Science and Technology (KAUST), CEMSE Division, Thuwal 23955-6900, Saudi Arabia. }
\email{levon.nurbekyan@kaust.edu.sa}



\thanks{	
	M. Cirant was partially supported by the Fondazione CaRiPaRo Project ``Nonlinear Partial Differential Equations: Asymptotic Problems and Mean-Field Games'', and by the Gruppo Nazionale per l'Analisi Matematica, la Probabilit\`a e le loro Applicazioni (GNAMPA) of the Istituto Nazionale di Alta Matematica (INdAM)}
\thanks{
	L. Nurbekyan was partially supported by King Abdullah University of Science and Technology (KAUST) baseline and start-up funds.}
\date{\today}

\begin{abstract}
Here, we observe that mean-field game (MFG) systems admit a two-player infinite-dimensional general-sum differential game formulation. We show that particular regimes of this game reduce to previously known variational principles. Furthermore, based on the game-perspective we derive new variational formulations for first-order MFG systems with congestion. Finally, we use these findings to prove the existence of time-periodic solutions for viscous MFG systems with a coupling that is not a non-decreasing function of density.
\end{abstract}

\maketitle

\tableofcontents

\section{Introduction}

In this note, we discuss the variational structure of mean-field game (MFG) systems and apply our findings to the construction of time-periodic solutions for systems with decreasing coupling.

MFG systems independently introduced by Lasry and Lions in \cite{LL061,LL062,LL07} and by Huang, Malham\'{e} and Caines \cite{HCM06,HCM07} is a framework to model populations that have a huge number of indistinguishable agents that play a differential game. In this framework, as in statistical physics, one models a huge population as a continuum of agents in some state space. Furthermore, the state of this population is modeled by the distribution of the agents in the state space. Hence, each agent in this game devises an optimal strategy based on the  distribution of the population; that is, on the statistical rather than individual information about the positions of other agents. As a result, one obtains a system of PDE, a MFG system, that characterizes the optimal actions of the agents and the evolution of their distribution in the state space. MFG systems are analogs of macroscopic equations from statistical physics in the game-theoretic framework.

Currently, MFG theory is a very active research direction with numerous applications in economics \cite{gll'11,moll,GNP}, finance \cite{gll'11,cl16,hjn15}, industrial engineering \cite{HCM06,HCM07}, crowd dynamics \cite{dogbe10}, knowledge growth \cite{blw16,blw17} and more. For further details on MFG theory we refer to \cite{LCDF,CardaNotes,gll'11,befreyam'13,Gomes2014,GoBook,cardela'18, notebari} and references therein.

From the PDE perspective, a typical MFG system has the form
\begin{equation}\label{eq:main}
\begin{cases}
-u_t-\epsi \Delta u + H\left(x, \nabla u,m\right)=0,\\
m_t-\epsi \Delta m - \mathrm{div}  \left(m \nabla_p H\left(x,\nabla u,m\right)\right)=0,\\
m> 0,\ m(x,0) = m^0(x),\ u(x,T)=u^T(x),~(x,t)\in \Tt^d \times [0,T].
\end{cases}
\end{equation}
Here, we denote by $\Tt^d$ the $d$-dimensional flat torus that is the state-space for a continuum of agents. Next, $m(\cdot,t),~0\leq t \leq T,$ denotes the density of the distribution of the agents in the state-space at time $t$. In this model, each agent faces a stochastic optimal control problem with an independent Brownian motion of intensity $\epsi \geq 0$, terminal time $T$, terminal cost function $u^T$, and a Lagrangian that depends on the distribution $m$ and gives rise to a Hamiltonian $H:\Tt^d \times \Rr^d \times \Rr_{++} \to \Rr$. The dependence of $H$ on $m$ is called the \textit{coupling} of the MFG system. Furthermore, $(x,t)\mapsto u(x,t)$ denotes the value function of this optimal control problem, and $m^0$ is the initial distribution density of the agents. We assume that $m^0>0$, and $\int_{\Tt^d} m^0=1$. The unknowns in \eqref{eq:main} are $u$ and $m$.

The first equation in \eqref{eq:main} is the Hamilton-Jacobi equation corresponding to the optimal control problem faced by each agent. The second equation in \eqref{eq:main} is the Fokker-Planck equation that governs the evolution of the density of the agents that act optimally. The ergodic version of \eqref{eq:main} is
\begin{equation}\label{eq:mainstationary}
\begin{cases}
-\epsi \Delta u + H(x, \nabla u,m)=\bar{H},\\
-\epsi \Delta m - \mathrm{div}  \left(m \nabla_p H(x,\nabla u,m)\right)=0,\\
m> 0,\ \int\limits_{\Tt^d} m(x)dx=1,~x\in \Tt^d.
\end{cases}
\end{equation}
This previous system corresponds to a model where agents solve an ergodic (long-time average) optimal control problem. In \eqref{eq:main} the unknowns are $u,m$ and $\Hh$. Latter, is the ergodic constant or the effective Hamiltonian and can be thought of as the Lagrange multiplier for the constraint $\int_{\Tt^d}m=1$.

We are interested in variational formulations of \eqref{eq:main} and \eqref{eq:mainstationary}; that is, we aim at finding functionals of $(m,u)$ that yield \eqref{eq:main} and \eqref{eq:mainstationary} as first-order optimality conditions.

The variational structure of MFG systems is not new. In a seminar paper on MFG, \cite{LL07}, Lasry and Lions pointed out that if $H$ has a  form
\begin{equation}\label{eq:H_separable}
H(x,p,m)=H_0(x,p)-f(x,m),\quad (x,p,m)\in \Tt^d \times \Rr^d\times \Rr_{++},
\end{equation}
for some $H_0:\Tt^d \times \Rr^d \to \Rr,~f:\Tt^d \times \Rr_{+} \to \Rr$ then \eqref{eq:main}, \eqref{eq:mainstationary} can be interpreted as optimality conditions for two dual optimal control problems of PDE -- \eqref{eq:FPcontrol} and \eqref{eq:HJcontrol}. Latter are reminiscent of the dynamical formulation of the Optimal Transportation Problem by Benamou and Brenier \cite{benbren'00}. These ideas were successfully used to analyze \eqref{eq:main}, \eqref{eq:mainstationary} in various settings where there are no direct regularizing mechanisms from elliptic and parabolic PDE theory. In particular, first-order systems were addressed in \cite{car'15,graber'14,CardaGraber,grames'18}, degenerate second-order systems were treated in \cite{CGPT}, and problems with constraints on $m$ were considered in \cite{MeSil,carmesa'16}. Moreover, the optimal control formulations in \cite{LL07} yield numerical solution methods for \eqref{eq:main}, \eqref{eq:mainstationary} by optimization techniques as in \cite{benbren'00}. We refer to \cite{bencar'15,bencarsan'17} and references therein for an account on these numerical methods. Additionally, similar ideas were used to treat mean-field type control problems with congestion in \cite{AL16,al'16b}.

Unfortunately, the optimal control formulation in \cite{LL07} does not extend to systems where $H$ is not of the form \eqref{eq:H_separable}. As a result, there is no systematic approach to analyze \eqref{eq:main}, \eqref{eq:mainstationary} when $H$ is not of the form \eqref{eq:H_separable} and there is no regularization due to ellipticity. For instance, suppose that
\begin{equation}\label{eq:H_congestion}
H(x,p,m)=m^{\alpha} H_0 \left(x,\frac{p}{m^\alpha}\right)-f(x,m),\quad (x,p,m)\in \Tt^d \times \Rr^d\times \Rr_{++},
\end{equation}
for some $H_0:\Tt^d \times \Rr^d \to \Rr,~f:\Tt^d \times \Rr_{+} \to \Rr$, and $\alpha > 0$. In this case, \eqref{eq:main}, \eqref{eq:mainstationary} correspond to so called \textit{soft congestion} models. In these models, agents pay more for moving in denser areas. The congestion strength is modeled by $\alpha$. Note that the no-congestion case, $\alpha=0$, corresponds to \eqref{eq:H_separable}.

As mentioned before, for $H$ as in \eqref{eq:H_congestion} there is no optimal control formulation of \eqref{eq:main}, \eqref{eq:mainstationary} analogous to the one in \cite{LL07}. Additionally, the singularity of $H$ at $m=0$ creates substantial difficulties when using purely PDE methods -- see \cite{govos'15,gomi'15,gra'15,evago'17,achpor'16} for second-order systems.

However, for a specific $H$ it may still be possible to find a non-standard variational formulation for \eqref{eq:main}, \eqref{eq:mainstationary}. In \cite{evafegonuvos'18}, for instance, authors found a new variational formulation for \eqref{eq:mainstationary} when $\epsi=0$ and
\begin{equation}\label{eq:H_congestion1}
H(x,p,m)=\frac{|p+Q|^\gamma}{\gamma m^{\alpha}}-f(x,m),\quad (x,p,m)\in \Tt^d \times \Rr^d\times \Rr_{++},
\end{equation}
for some $f:\Tt^d \times \Rr_{++} \to \Rr,~Q\in \Rr^d$, and $1< \alpha \leq \gamma$. The formulation in \cite{evafegonuvos'18} yielded well-defined variational solutions that are unique and can be numerically calculated by optimization techniques.

In \cite{alfego'17,FG2}, authors used yet another variational approach to MFG systems that is closely related to the uniqueness of solutions for \eqref{eq:main}, \eqref{eq:mainstationary}. In \cite{LCDF}, Lions derived a sufficient condition that guarantees uniqueness of smooth solutions for \eqref{eq:main}, \eqref{eq:mainstationary}. The condition reads as
\begin{equation}\label{eq:Lionsuniq}
\begin{pmatrix}
2 \nabla^2_{p} H(x,p,m)&\frac{\partial }{\partial m} \nabla_p H(x,p,m)\\
\frac{\partial }{\partial m} \nabla^\top_p H(x,p,m)&-\frac{2}{m}\frac{\partial H(x,p,m)}{\partial m }
\end{pmatrix}\geq 0,
\end{equation}
for all $(x,p,m)\in \Tt^d \times \Rr^d \times \Rr_{++}$. In \cite{alfego'17,FG2}, Gomes and coauthors observed that under \eqref{eq:Lionsuniq} the map
\begin{equation*}
A:\begin{pmatrix}
m\\
u
\end{pmatrix}
\mapsto
\begin{pmatrix}
u_t+\epsi \Delta u - H(x,\nabla u,m)\\
m_t-\epsi \Delta m - \mathrm{div}  \left(m \nabla_p H(x,\nabla u,m)\right)
\end{pmatrix}
\end{equation*}
is a monotone operator. Therefore, \eqref{eq:mainstationary} reduces to finding the zeros of a monotone operator. Furthermore, in \cite{alfego'17,FG2} authors defined weak solutions of \eqref{eq:mainstationary} using variational inequalities and proved their existence using Minty's method from monotone operators theory. Moreover, Gomes and coauthors developed numerical solution methods for \eqref{eq:mainstationary} and finite-state version of \eqref{eq:main} in \cite{alfego'17} and \cite{GJS2}, respectively, using monotone flows. For an overview of monotonicity methods for MFG systems we refer to \cite{gomesIPAM'17}.

In this note, we aim at developing a systematic way to search for non-standard variational formulations for MFG systems. We formally show that solutions of \eqref{eq:main} and \eqref{eq:mainstationary} correspond to critical points of suitable functionals. Furthermore, we observe that under standard monotonicity assumptions these critical points correspond to Nash equilibria of two-player infinite-dimensional differential games \eqref{eq:game}, \eqref{eq:gamestat}, \eqref{eq:gameLagrange}: Propositions \ref{prp:Nashnonzero}, \ref{prp:Nashstat}, \ref{prp:NashstatLagrange}.

Interestingly, when $H$ is of the form \eqref{eq:H_separable} we obtain zero-sum games, Corollaries \ref{crl:Nash_a=0_Hseparated}, \ref{crl:Nashstat_a=0_Hseparated}, \ref{crl:Nashstat_a=0_HseparatedLagrange}, that are directly linked to the optimal control formulations in \cite{LL07}. More precisely, in Section \ref{sec:previous} we observe that these zero-sum game formulations are the Hamiltonian forms of the optimal control problems in \cite{LL07}.

For stationary first-order systems with $H$ of the form \eqref{eq:H_congestion1} and parameter range $0<\alpha<1 \leq \gamma$ we again obtain a zero-sum game and thus find a saddle point formulation of \eqref{eq:mainstationary} that is new in the literature: Corollaries \ref{crl:Nashstat_congestion}, \ref{crl:NashstatLagrange_congestion}. Furthermore, in analogy with \cite{benbren'00}, we find a convex optimization formulation of this saddle point problem that is also new: Remark \ref{rmk:BBformcong_a<1}. Moreover, as we observe in Section \ref{sec:previous}, this convex optimization formulation is a generalization of the transformation in \cite{evafegonuvos'18} that was used to solve \eqref{eq:mainstationary} for $H$ of the form \eqref{eq:H_congestion1} and parameter range $0<\alpha<1\leq \gamma$ in the two-dimensional case.

For stationary first-order systems with $H$ of the form \eqref{eq:H_congestion1} and parameter range $1< \alpha,~1 \leq \gamma$ we recover a potential game instead of a zero-sum game, Corollaries \ref{crl:Nashstat_congestion_a>1}, \ref{crl:NashstatLagrange_congestion_a>1} and Remark \ref{rmk:potGame_a>1}. Furthermore, when $1< \alpha \leq \gamma$ the potential of this game is concave, and therefore Nash equilibria are the maximizers of this potential. Thus, in Section \ref{sec:previous}, we recover the variational principle in \cite{evafegonuvos'18}.

Additionally, we discuss some interpretations of our results in terms of the mean-field type or McKean-Vlasov control theory: Remarks \ref{rmk:socialcost}, \ref{rmk:Vlasov}, \ref{rmk:varstructure}.

We would like to stress that our discussion on differential-game formulations for \eqref{eq:main} and \eqref{eq:mainstationary} are formal; that is, we perform our analysis at the level of smooth functions. Accordingly, we do not address the existence and regularity theory of suitable weak solutions. These are extremely interesting and apparently challenging problems.

In the final section of this work, we propose an application of the differential-game formulation of \eqref{eq:main} to systems where the uniqueness condition \eqref{eq:Lionsuniq} is violated. We aim in particular at finding time-periodic solutions in the special case $H(x, p, m) = |p|^2/2 - f(m)$, where $f$ is a decreasing function. Previously, a similar result with completely different techniques was obtained in \cite{gosed'17} for a congestion model $H(x,p,m)=\frac{|p\pm \sqrt{2K}|^2}{2m}+K m,~K>0$.

Recently, an increasing interest has been devoted to the study of models where multiplicity of equilibria arises \cite{BarFis, CCes, c16, cpre, Go16,gosed'17,USG} (also in the multi-population setting, see e.g. \cite{ABC, CV, LW}). In a previous work \cite{CirO}, it has been proven the existence of solutions in simple models with separable Hamiltonians that exhibit an oscillatory behaviour on $[0, T]$. Such results have been obtained by means of bifurcation methods, looking in particular at branches of non-trivial solutions as the time horizon $T$ varies. The arise of periodic patterns led naturally to the question of existence of truly time-periodic solutions, namely a couple $m, u$ defined for all $t \in (-\infty, +\infty)$ and such that $m(\cdot, t + T) = m(\cdot, t)$ for some $T > 0$ and for all $t$. Here, we provide a positive answer to this question (see Theorem \ref{thm:bifu}), under the assumption that $-f'(1)$ is large enough.

The result again relies on bifurcation techniques, that in turn exploit an analysis of the linearized system similar to the one in \cite{CirO}. A crucial point is that such techniques require that solutions to the linearized system consist of a vector space of dimension one (or odd); here, we overcome this issue using the variational structure of \eqref{eq:main}, that enables us to implement bifurcation for potential operators (namely Theorem \ref{bifu}). These are not restricted to odd dimension of eigenspaces. We finally observe that the equilibria that we find have non-trivial dependance in time (see Remark \ref{eq:main}); this is remarkable in view of the autonomous nature of \eqref{eq:main}, where no periodic force acts explicitly.

The paper is organized as follows. In Section \ref{sec:notation} we introduce the notation and hypotheses. Section \ref{sec:dgame} is devoted to the two-player infinite-dimensional game formulations \eqref{eq:game}, \eqref{eq:gamestat}, \eqref{eq:gameLagrange} of \eqref{eq:main} and \eqref{eq:mainstationary} for a general Hamiltonian. Furthermore, in Section \ref{sec:regimes} we discuss particular Hamiltonians for which these games have more structure. Next, in Section \ref{sec:previous} we discuss the connections among \eqref{eq:game}, \eqref{eq:gamestat}, \eqref{eq:gameLagrange} and previously known variational principles. Finally, in Section \ref{sec:bifu} we prove the existence of time-periodic solutions for \eqref{eq:main} when \eqref{eq:Lionsuniq} is violated.

\section{Notation and assumptions}\label{sec:notation}

We denote by $\Tt^d,~\Rr^d,~\Rr_{+},~\Rr_{++}$, respectively, the $d$-dimensional flat torus, the $d$-dimensional Euclidean space, the set of nonnegative real numbers and the set of positive real numbers. For $T>0$ we denote by
\begin{equation*}
\Omega_T=\Tt^d \times [0,T].
\end{equation*}
For real positive and non-integer $\beta$, $C^{\beta, \beta/2}(\Omega_T)$ will be the standard H\"older parabolic space.

Throughout the note we assume that $H\in C^2(\Tt^d\times \Rr^d\times \Rr_{++})$ and $m^0,u^T \in C^2(\Tt^d),~m^0>0$. Furthermore, for the differential-game formulation of \eqref{eq:main} and \eqref{eq:mainstationary} we need the following monotonicity assumption.
\begin{hyp}
For all $(x,p,m) \in \Tt^d\times \Rr^d\times \Rr_{++}$ one has that
\begin{equation}\label{hyp:mon_convexity}
\nabla_{pp}^2 H(x,p,m) \geq 0,\quad \partial_m H(x,p,m) \leq 0.
\end{equation}
\end{hyp}
Note that \eqref{hyp:mon_convexity} is necessary for \eqref{eq:Lionsuniq}. Furthermore, we define $F_{H}$ by
\begin{equation}\label{eq:F_H}
F_H(x,p,m)=\int\limits^m H(x,p,z)dz,\quad (x,p,m)\in \Tt^d\times \Rr^d\times \Rr_{++}.
\end{equation}

For arbitrary functional $n \mapsto K(n)$ we denote by $\frac{\delta K}{\delta n}$ its variational derivative.

\section{A two-player infinite-dimensional differential game}\label{sec:dgame}

In this section, we show that under \eqref{hyp:mon_convexity} the smooth solutions of \eqref{eq:main} and \eqref{eq:mainstationary} are Nash equilibria of a suitable two-player infinite-dimensional  differential game. For $(m,u)\in C^2(\Omega_T),~m>0$ consider the following functionals:
\begin{equation}\label{eq:Psi1}
\begin{split}
\Psi_1(m,u)=&\int\limits_{\Omega_T} m(-u_t-\epsi \Delta u)+F_{H}(x,\nabla u,m)dxdt+\int\limits_{\Tt^d}m(x,T)u(x,T)dx\\
&-\int\limits_{\Tt^d} m^0(x)u(x,0)dx-\int\limits_{\Tt^d}m(x,T)u^T(x)dx\\
=&\int\limits_{\Omega_T} u(m_t-\epsi \Delta m)+F_{H}(x,\nabla u,m)dxdt+\int\limits_{\Tt^d}m(x,0)u(x,0)dx\\
&-\int\limits_{\Tt^d} m^0(x)u(x,0)dx-\int\limits_{\Tt^d}m(x,T)u^T(x)dx,
\end{split}
\end{equation}
and
\begin{equation}\label{eq:Psi2}
\begin{split}
\Psi_2(m,u)=&\int\limits_{\Omega_T} m(-u_t-\epsi \Delta u)+m H(x,\nabla u,m)dxdt+\int\limits_{\Tt^d}m(x,T)u(x,T)dx\\
&-\int\limits_{\Tt^d} m^0(x)u(x,0)dx-\int\limits_{\Tt^d}m(x,T)u^T(x)dx\\
=&\int\limits_{\Omega_T} u(m_t-\epsi \Delta m)+m H(x,\nabla u,m)dxdt+\int\limits_{\Tt^d}m(x,0)u(x,0)dx\\
&-\int\limits_{\Tt^d} m^0(x)u(x,0)dx-\int\limits_{\Tt^d}m(x,T)u^T(x)dx.
\end{split}
\end{equation}

Now, we consider the following differential game:
\begin{equation}\label{eq:game}
\begin{split}
\mbox{Player 1}&\qquad\qquad \sup\limits_{m\in C^2(\Omega_T),~m>0} \Psi_1(m,u)\\
&\\
\mbox{Player 2}&\qquad\qquad \inf\limits_{u\in C^2(\Omega_T)} \Psi_2(m,u).
\end{split}
\end{equation}
In \eqref{eq:game}, the first player chooses a strategy $m\in C^2(\Omega_T),~m>0$ and aims at maximizing $\Psi_1(m,u)$. The second player chooses a strategy $u\in C^2(\Omega_T)$ and aims at minimizing $\Psi_2(m,u)$. Note that $\Psi_1$ and $\Psi_2$ are different in general, and hence \eqref{eq:game} is a general-sum game.

\begin{proposition}\label{prp:Nashnonzero}
Suppose that \eqref{hyp:mon_convexity} holds. Then, for any $m^*,u^* \in C^2(\Omega_T)$ such that $m^*>0$ one has that
\begin{equation}\label{eq:Psi1ineq}
\Psi_1(m^*,u^*) \geq \Psi_1(m,u^*),~ \mbox{for all}~m\in C^2(\Omega_T),~m>0,
\end{equation}
and
\begin{equation}\label{eq:Psi2ineq}
\Psi_2(m^*,u^*) \leq \Psi_2(m^*,u),~ \mbox{for all}~u\in C^2(\Omega_T),
\end{equation}
if and only if $(m^*,u^*)$ is a classical solution of \eqref{eq:main}.
\end{proposition}
\begin{proof}
A straightforward calculation of the variational derivatives of $\Psi_1, \Psi_2$ yields that \eqref{eq:main} can be written as
\begin{equation*}
\begin{cases}
\frac{\delta \Psi_1}{\delta m}(m,u)=0,~m>0,\\
\frac{\delta \Psi_2}{\delta u}(m,u)=0.
\end{cases}
\end{equation*}
Then, the proof follows from the fact that $m\mapsto \Psi_1(m,u),~m>0$ is a concave functional, and $u\mapsto \Psi_2(m,u)$ is a convex functional when $m>0$. Although elementary, we present the proof here for the sake of completeness.

We start by the direct implication; that is, we assume that $(m^*,u^*),~m^*>0$ is a solution of \eqref{eq:main}, and we aim at proving \eqref{eq:Psi1ineq}, \eqref{eq:Psi2ineq}. Pick an arbitrary test function $m \in C^2(\Omega_T),~m>0$ and consider
\begin{equation}\label{eq:I1}
I_1(h)=\Psi_1(m^*+h\phi,u^*)=\Psi_1(m_h,u^*),
\end{equation}
where $\phi=m-m^*$. Note that $I_1$ is well defined in some neighborhood of $h=0$ that contains $h=1$ because $m,~m^*>0$. Then, we have that
\begin{equation*}
\begin{split}
\frac{dI_1(h)}{dh}=&\int\limits_{\Omega_T} \phi \left(-u^*_t-\epsi \Delta u^*+H\left(x,\nabla u^*,m_h\right)\right)dxdt+\int\limits_{\Tt^d}\phi(x,T)u^*(x,T)dx\\
&-\int\limits_{\Tt^d}\phi(x,T)u^T(x)dx\\
=&\int\limits_{\Omega_T} \phi \left(-u^*_t-\epsi \Delta u^*+H\left(x,\nabla u^*,m_h\right)\right)dxdt.
\end{split}
\end{equation*}
Therefore, we have that
\begin{equation*}
I'_1(0)=\int\limits_{\Omega_T} \phi \left(-u^*_t-\epsi \Delta u^*+ H\left(x,\nabla u^*,m^*\right)\right)dxdt=0,
\end{equation*}
because $(m^*,u^*)$ is a solution of \eqref{eq:main}. Furthermore, we have that
\begin{equation*}
\begin{split}
\frac{d^2I_1(h)}{dh^2}=\int\limits_{\Omega_T} \phi^2 \frac{\partial H\left(x,\nabla u^*,m\right)}{\partial m} \bigg|_{m_h}dxdt \leq 0,
\end{split}
\end{equation*}
by \eqref{hyp:mon_convexity}. Hence, $h\mapsto I_1(h)$ is a concave function, and its critical points are points of maxima. Therefore, $I_1(1) \leq I_1(0)$ which is exactly \eqref{eq:Psi1ineq}.

Next, we prove \eqref{eq:Psi2ineq}. For arbitrary $u \in C^2(\Omega_T)$ consider
\begin{equation}\label{eq:I2}
I_2(h)=\Psi_2(m^*,u^*+h\psi)=\Psi_2(m^*,u_h),
\end{equation}
where $\psi=u-u^*$. Then, we have that
\begin{equation*}
\begin{split}
\frac{dI_2(h)}{dh}=&\int\limits_{\Omega_T} \psi (m^*_t-\epsi \Delta m^*)+m^* \nabla_pH\left(x,\nabla u_h,m^*\right)\cdot \nabla \psi dxdt+\int\limits_{\Tt^d}m^*(x,0)\psi(x,0)dx\\
&-\int\limits_{\Tt^d}m^0(x)\psi(x,0)dx\\
=&\int\limits_{\Omega_T} \psi \left(m^*_t-\epsi \Delta m^*-\div \left(m^* \nabla_p H\left(x, \nabla u_h,m^*\right)\right)\right) dxdt.
\end{split}
\end{equation*}
Hence, we obtain that
\begin{equation*}
I_2'(0)=\int\limits_{\Omega_T} \psi \left(m^*_t-\epsi \Delta m^*-\div \left(m^* \nabla_pH\left(x,\nabla u^*,m^*\right)\right)\right) dxdt=0,
\end{equation*}
because $(m^*,u^*)$ is a solution of \eqref{eq:main}. Furthermore, we have that
\begin{equation*}
\begin{split}
\frac{d^2I_2(h)}{dh^2}=&\int\limits_{\Omega_T}  \nabla \psi\cdot \nabla_{pp}^2H\left(x,\nabla u_h,m^*\right)\nabla \psi dxdt \geq 0,
\end{split}
\end{equation*}
by \eqref{hyp:mon_convexity}. Hence, $h\mapsto I_2(h)$ is a convex function, and its critical points are minima. Thus, $I_2(1) \geq I_2(0)$ which yields \eqref{eq:Psi2ineq}.

Now, we turn to the inverse implication. Pick an arbitrary test function $\phi\in C^2(\Omega_T)$ and consider $I_1$ given by \eqref{eq:I1}. Then, $h\mapsto I_1(h)$ is defined in the neighborhood of $h=0$ because $m^*>0$. Since, $(m^*,u^*)$ satisfies \eqref{eq:Psi1ineq} we have that $I_1(h) \leq I_1(0)$ for $h \in \mathrm{dom}~I_1$. Hence, we have that $I_1'(0)=0$. From the previous calculations we have that
\begin{equation*}
\begin{split}
I_1'(0)=&\int\limits_{\Omega_T} \phi \left(-u^*_t-\epsi \Delta u^*+ H\left(x,\nabla u^*,m^*\right)\right)dxdt\\
&+\int\limits_{\Tt^d}\phi(x,T)u^*(x,T)dx-\int\limits_{\Tt^d}\phi(x,T)u^T(x)dx.
\end{split}
\end{equation*}
Since $\phi$ is arbitrary we obtain
\begin{equation*}
\begin{cases}
-u^*_t-\epsi \Delta u^*+ H\left(x,\nabla u^*,m^*\right)=0,\\
u^*(x,T)=u^T(x),~(x,t)\in \Omega_T.
\end{cases}
\end{equation*}
Furthermore, let $\psi\in C^2(\Omega_T)$ and consider $I_2$ given by \eqref{eq:I2}. Then, we have that $I_2(h)\geq I_2(0),~h\in \Rr$ by \eqref{eq:Psi2ineq}. Consequently, we obtain that $I_2'(0)=0$. But we have that
\begin{equation*}
\begin{split}
I_2'(0)=&\int\limits_{\Omega_T} \psi \left(m^*_t-\epsi \Delta m^*-\div \left(m^* \nabla_pH\left(x,\nabla u^*,m^*\right)\right)\right) dxdt\\
&+\int\limits_{\Tt^d}m^*(x,0)\psi(x,0)dx-\int\limits_{\Tt^d}m^0(x)\psi(x,0)dx=0.
\end{split}
\end{equation*}
Since $\psi$ is arbitrary, we obtain that
\begin{equation*}
\begin{cases}
m^*_t-\epsi \Delta m^*-\div \left(m^* \nabla_pH\left(x,\nabla u^*,m^*\right)\right)=0,\\
m^*(x,0)=m^0(x),~(x,t)\in \Omega_T,
\end{cases}
\end{equation*}
and the proof is complete.
\end{proof}
For the ergodic version of \eqref{eq:game} we introduce the payoff functionals
\begin{equation}\label{eq:Psi1stat}
\begin{split}
\hat{\Psi}_1(m,u)=&\int\limits_{\Tt^d} - \epsi m  \Delta u+F_{H}(x,\nabla u,m)dx\\
=&\int\limits_{\Tt^d} -\epsi u \Delta m+F_{H}(x,\nabla u,m)dx,
\end{split}
\end{equation}
and
\begin{equation}\label{eq:Psi2stat}
\begin{split}
\hat{\Psi}_2(m,u)=&\int\limits_{\Tt^d}  -\epsi m \Delta u+m H\left(x,\nabla u,m\right)dx\\
=&\int\limits_{\Tt^d} -\epsi u \Delta m+m H\left(x,\nabla u,m\right)dx,
\end{split}
\end{equation}
for $m,u \in C^2(\Tt^d),~m>0$. Then, we consider the following game:
\begin{equation}\label{eq:gamestat}
\begin{split}
\mbox{Player 1}&\qquad\qquad \sup\left\{\hat{\Psi}_1(m,u):~m \in C^2(\Tt^d),~m>0,~\int\limits_{\Tt^d}m=1\right\}\\
&\\
\mbox{Player 2}&\qquad\qquad \inf\left\{\hat{\Psi}_2(m,u):~u \in C^2(\Tt^d)\right\}.
\end{split}
\end{equation}

\begin{proposition}\label{prp:Nashstat}
Suppose that \eqref{hyp:mon_convexity} holds. Then, for any $m^*,u^* \in C^{2}(\Tt^d)$ such that $m^*>0$ and $\int\limits_{\Tt^d}m^*=1$ one has that
\begin{equation}\label{eq:Psi1ineqstat}
\hat{\Psi}_1(m^*,u^*) \geq \hat{\Psi}_1(m,u^*),~ \mbox{for all}~m\in C^{2}(\Tt^d),~m>0,~\int\limits_{\Tt^d}m=1.
\end{equation}
and
\begin{equation}\label{eq:Psi2ineqstat}
\hat{\Psi}_2(m^*,u^*) \leq \hat{\Psi}_2(m^*,u),~ \mbox{for all}~u\in C^{2}(\Tt^d),
\end{equation}
if and only if $(m^*,u^*)$ is a classical solution of \eqref{eq:mainstationary} for some $\Hh^* \in \Rr$.

Moreover, for both cases above one has that
\begin{equation}\label{eq:Hstarstat}
\Hh^*=\hat{\Psi}_2(m^*,u^*).
\end{equation}
\end{proposition}
\begin{proof}
The proof is analogous to the one of Proposition \ref{prp:Nashnonzero}.
\end{proof}

\begin{remark}\label{rmk:socialcost}
As we pointed out in the Introduction, \eqref{eq:main} and \eqref{eq:mainstationary} are Nash equilibrium conditions in a differential game with infinitely many agents that interact through the distribution of the whole population, the mean-field. Interestingly, $\Psi_2$ and $\hat{\Psi}_2$ can be interpreted in terms of the average payoff of the population. More precisely, for a given control $r:\Omega_T\mapsto \Rr^d$ consider the functional
\begin{equation}\label{eq:Scost}
\mathcal{S}(r)= \int\limits_{\Omega_T} L\big(x,-r(x,s),m(x,s)\big) m(x,s)dxds+\int\limits_{\Tt^d} u^T(x) m(x,T) dx,
\end{equation}
where $L(x,v,m)=\sup\limits_{p} v\cdot p - H(x,p,m)$, and the dynamics is given by
\begin{equation}\label{eq:rcontrol}
\begin{cases}
m_t-\epsi \Delta m+\div (m r)=0,\\
m(x,0)=m^0(x),~(x,t) \in \Omega_T.
\end{cases}
\end{equation}
The mean-field type or McKean-Vlasov optimal control problem reduces to
\begin{equation}\label{eq:Vlasovcontrol}
\inf \{\mathcal{S}(r)~\mbox{s.t.}~\eqref{eq:rcontrol}~\mbox{holds}\}.
\end{equation}
Here, $\mathcal{S}(r)$ is interpreted as average payoff or the social cost of the population when control $r$ is applied. Thus, \eqref{eq:Vlasovcontrol} is an optimal control problem for a central planner that aims at minimizing the average cost per agent. We refer to \cite{cardela'18} for a detailed discussion on McKean-Vlasov optimal control problems and their relation to MFG systems.

Suppose that $m,u\in C^2(\Omega_T),~m>0$ are such that $m(x,0)=m^0(x),u(x,T)=u^T(x),~x\in \Tt^d$, and
\begin{equation}\label{eq:FP}
m_t-\epsi \Delta m-\div (m \nabla_p H(x,\nabla u,m))=0,~(x,t) \in \Omega_T.
\end{equation}
Then, $t\mapsto m(\cdot,t)$ evolves according to \eqref{eq:rcontrol}, where the control $r$ is given by
\begin{equation}\label{eq:roptimal}
r_{m,u}(x,t)=-\nabla_p H(x,\nabla u(x,t),m(x,t)),~(x,t)\in \Omega_T. 
\end{equation}
Thus, using the identity
\begin{equation*}
L(x,\nabla_p H(x,p,m),m )=p\cdot \nabla_p H(x,p,m)-H(x,p,m)
\end{equation*}
we obtain that
\begin{equation}\label{eq:SPsi2}
\begin{split}
\mathcal{S}(r_{m,u})=& \int\limits_{\Omega_T} L\big(x,-r_{m,u}(x,s),m(x,s)\big) m(x,s)dxds+\int\limits_{\Tt^d} u^T(x) m(x,T) dx\\
=& \int\limits_{\Omega_T}  \big(\nabla u \cdot \nabla_p H(x,\nabla u ,m)-H(x,\nabla u,m)\big) m dxds+\int\limits_{\Tt^d} u^T(x) m(x,T) dx\\
=& \int\limits_{\Omega_T}  - u \cdot \mathrm{div} \left(m \nabla_p H(x,\nabla u ,m)\right)-H(x,\nabla u,m)  m dxds+\int\limits_{\Tt^d} u^T(x) m(x,T) dx\\
=& \int\limits_{\Omega_T}   - u (m_t-\epsi \Delta m)-H(x,\nabla u,m)  m dxds+\int\limits_{\Tt^d} u^T(x) m(x,T) dx\\
=&-\Psi_2(m,u).
\end{split}
\end{equation}
Now, suppose that $(m^*,u^*) \in C^2(\Omega_T),~m^*>0$ is a Nash equilibrium of \eqref{eq:game}. Then, by Proposition \ref{prp:Nashnonzero} $(m^*,u^*)$ is a classical solution of \eqref{eq:main}, and hence \eqref{eq:FP} is valid for $(m,u)=(m^*,u^*)$. Therefore, from \eqref{eq:SPsi2} one has that the social cost corresponding to the MFG equilibrium is equal to $-\Psi_2(m^*,u^*)$; that is, the negative of the Player 2 value in \eqref{eq:game}.

Similarly, from Proposition \ref{prp:Nashstat} we have that if $(m^*,u^*) \in C^2(\Tt^d),~m^*>0$ is a Nash equilibrium in \eqref{eq:gamestat}, then $(m^*,u^*,\Hh^*)$ is a classical solution of \eqref{eq:mainstationary} for some $\Hh^* \in \Rr$, and  $-\hat{\Psi}_2(m^*,u^*)$ is the ergodic social cost. But we have that $-\hat{\Psi}_2(m^*,u^*)=-\Hh^*$. Hence, $-\Hh^*$ in \eqref{eq:mainstationary} can be interpreted as an ergodic social cost.
\end{remark}
\begin{remark}\label{rmk:Vlasov}
The PDE system corresponding to the optimality conditions of \eqref{eq:Vlasovcontrol} is given by
\begin{equation}\label{eq:Vmain}
\begin{cases}
-u_t-\epsi \Delta u + H\left(x, \nabla u,m\right)+m \partial_m H\left(x, \nabla u,m\right)=0,\\
m_t-\epsi \Delta m - \mathrm{div}  \left(m \nabla_p H\left(x,\nabla u,m\right)\right)=0,\\
m> 0,\ m(x,0) = m^0(x),\ u(x,T)=u^T(x),~(x,t)\in \Tt^d \times [0,T].
\end{cases}
\end{equation}
Suppose that $(\hat{m},\hat{u})\in C^2(\Omega_T)$ is a solution of this system. Then $r_{\hat{m},\hat{u}}$ given by \eqref{eq:roptimal} is a solution of \eqref{eq:Vlasovcontrol} \cite{befreyam'13,cardela'18}. Thus, if $(m^*,u^*)\in C^2(\Omega_T)$ is a solution of \eqref{eq:main}, and $r_{m^*,u^*}$ is the corresponding control given by \eqref{eq:roptimal}, then
\begin{equation*}
\mathcal{S}(r_{\hat{m},\hat{u}}) \leq \mathcal{S}(r_{m^*,u^*}).
\end{equation*}
Furthermore, using \eqref{eq:SPsi2} we get that
\begin{equation}\label{eq:Nashsuboptimal}
\Psi_2(m^*,u^*) \leq \Psi_2 (\hat{m},\hat{u}).
\end{equation}
Interestingly, \eqref{eq:Nashsuboptimal} can be obtained in a simple and direct manner without going through McKean-Vlasov or PDE optimal control approach. Indeed, one can check that smooth solutions of \eqref{eq:Vmain} are the critical points of the functional $\Psi_2$. Hence, similar to Proposition \ref{prp:Nashnonzero}, if $p \mapsto H(x,p,m)$ is convex and $m \mapsto m H(x,p,m)$ is concave, then smooth solutions of \eqref{eq:Vmain} are Nash equilibria of a two-player zero-sum game
\begin{equation}\label{eq:Vgame}
\begin{split}
\mbox{Player 1}&\qquad\qquad \sup\limits_{m\in C^2(\Omega_T),~m>0} \Psi_2(m,u)\\
&\\
\mbox{Player 2}&\qquad\qquad \inf\limits_{u\in C^2(\Omega_T)} \Psi_2(m,u).
\end{split}
\end{equation}
Now, suppose that $(m^*,u^*)$ and $(\hat{m},\hat{u})$ are Nash equilibria for \eqref{eq:game} and \eqref{eq:Vgame} respectively. Then, $\hat{u}$ is a suboptimal strategy for Player 2 in \eqref{eq:game} so
\begin{equation*}
\Psi_2(m^*,u^*) \leq \Psi_2(m^*,\hat{u}).
\end{equation*}
Furthermore, $m^*$ is a suboptimal strategy for Player 1 in \eqref{eq:Vgame} so
\begin{equation*}
\Psi_2(m^*,\hat{u}) \leq \Psi_2(\hat{m},\hat{u}),
\end{equation*}
and we arrive at \eqref{eq:Nashsuboptimal}.
\end{remark}

\begin{remark}\label{rmk:varstructure}
In contrast with \eqref{eq:main} and \eqref{eq:mainstationary}, system \eqref{eq:Vmain} always admits a variational formulation \cite{AL16}. One way to explain this phenomenon is that \eqref{eq:main} and \eqref{eq:mainstationary} correspond to Nash equilibria of a nonzero-sum game and \eqref{eq:Vmain} corresponds to a zero-sum game. Moreover, when $H$ is separable \eqref{eq:H_separable} we observe in Corollaries \ref{crl:Nash_a=0_Hseparated}, \ref{crl:Nashstat_a=0_Hseparated} that games \eqref{eq:game}, \eqref{eq:gamestat} are equivalent to zero-sum games \eqref{eq:gameseparable}, \eqref{eq:gameseparablestat}. Thus, it is expected that \eqref{eq:main} and \eqref{eq:mainstationary} admit variational formulations for separable $H$. This is indeed the case as observed in \cite{LL07}. 
\end{remark}

In Proposition \ref{prp:Nashstat}, the first player faces an optimization problem with constraint $\int_{\Tt^d} m=1$ that yields a Lagrange multiplier $\Hh^*$. Moreover, $\Hh^*$ is the value of the second player in the equilibrium. In fact, one can incorporate $\Hh^*$ in the differential game. For that, we define
\begin{equation}\label{eq:Psi12statLagrange}
\begin{split}
\tilde{\Psi}_1(m,u,\Hh)=&\int\limits_{\Tt^d} - \epsi m  \Delta u+F_{H}(x,\nabla u,m)+\Hh(1-m)dx,\\
\tilde{\Psi}_2(m,u,\Hh)=&\int\limits_{\Tt^d}  -\epsi m \Delta u+m H\left(x,\nabla u,m\right)+\Hh(1-m)dx,
\end{split}
\end{equation}
for $m,u \in C^2(\Tt^d),~m>0,~\Hh\in \Rr$ and consider the game
\begin{equation}\label{eq:gameLagrange}
\begin{split}
\mbox{Player 1}&\qquad\qquad \sup\limits_{m\in C^2(\Tt^d),~m>0} \tilde{\Psi}_1(m,u,\Hh)\\
&\\
\mbox{Player 2}&\qquad\qquad \inf\limits_{u\in C^2(\Tt^d),~\Hh\in \Rr} \tilde{\Psi}_2(m,u,\Hh).
\end{split}
\end{equation}
Then, we have the following proposition.
\begin{proposition}\label{prp:NashstatLagrange}
Suppose that \eqref{hyp:mon_convexity} holds. Then, for any $m^*,u^* \in C^{2}(\Tt^d),~\Hh^*\in \Rr$ such that $m^*>0$ one has that
\begin{equation}\label{eq:Psi1ineqstatLagrange}
\tilde{\Psi}_1(m^*,u^*,\Hh^*) \geq \tilde{\Psi}_1(m,u^*,\Hh^*),~ \mbox{for all}~m\in C^{2}(\Tt^d),~m>0.
\end{equation}
and
\begin{equation}\label{eq:Psi2ineqstatLagrange}
\tilde{\Psi}_2(m^*,u^*,\Hh^*) \leq \tilde{\Psi}_2(m^*,u,\Hh),~ \mbox{for all}~(u,\Hh)\in C^{2}(\Tt^d)\times \Rr,
\end{equation}
if and only if $(m^*,u^*,\Hh^*)$ is a classical solution of \eqref{eq:mainstationary}. 
\end{proposition}
\begin{proof}
The proof is analogous to the one of Proposition \ref{prp:Nashnonzero}. The key point is that the $m \mapsto \tilde{\Psi}_1(m,u^*,\Hh^*)$ is a concave functional, and $(u,\Hh) \mapsto \tilde{\Psi}_2(m^*,u,\Hh)$ is a convex functional.
\end{proof}

\section{Various regimes of the game}\label{sec:regimes}

In this section, we consider several types of $H$ for which \eqref{eq:game} and \eqref{eq:gamestat} can be simplified. In particular, we consider separable and power-like Hamiltonians with congestion.  

\subsection{Separable Hamiltonian}

In this section, we assume that $H$ is of the form \eqref{eq:H_separable} for some $H_0 \in C^2(\Tt^d \times \Rr^d)$ and $f\in C^1(\Tt^d\times \Rr_{++})$. In this case, \eqref{eq:main} and \eqref{eq:mainstationary} respectively become
\begin{equation}\label{eq:main_a=0_Hseparated}
\begin{cases}
-u_t-\epsi \Delta u + H_0\left(x, \nabla u\right)=f(x,m),\\
m_t-\epsi \Delta m - \mathrm{div}  \left(m \nabla_p H_0\left(x,\nabla u\right)\right)=0,\\
m> 0,\ m(x,0) = m^0(x),\ u(x,T)=u^T(x),
\end{cases}
\end{equation}
and
\begin{equation}\label{eq:mainstat_a=0_Hseparated}
\begin{cases}
-\epsi \Delta u + H_0(x, \nabla u)=f(x,m)+\bar{H},\\
-\epsi \Delta m - \mathrm{div}  \left(m \nabla_p H_0(x,\nabla u)\right)=0,\\
m> 0,\ \int\limits_{\Tt^d} m(x)dx=1.
\end{cases}
\end{equation}
The assumption \eqref{hyp:mon_convexity} in this case is equivalent to the following one.
\begin{hyp}
For all $(x,p,m) \in \Tt^d\times \Rr^d \times \Rr_{++}$ one has that
\begin{equation}\label{hyp:mon_convexity_a=0_Hseparated}
\nabla_{pp}^2 H_0(x,p) \geq 0,\quad \partial_m f(x,m) \geq 0.
\end{equation}
\end{hyp}
Furthermore, $F_{H}$ in \eqref{eq:F_H} is given by
\begin{equation*}
F_H(x,p,m)=m H_0(x,p) - F(x,m),\quad (x,p,m)\in \Tt^d \times \Rr^d \times \Rr_{++}.
\end{equation*}
where $F(x,m)=\int\limits^{m} f(x,z)dz$. Hence, $\Psi_1$ in \eqref{eq:Psi1} has the form
\begin{equation}\label{eq:Psi1_a=0_Hseparated}
\begin{split}
\Psi_1(m,u)=&\int\limits_{\Omega_T} m(-u_t-\epsi \Delta u)+m H_0(x,\nabla u)- F(x,m)dxdt+\int\limits_{\Tt^d}m(x,T)u(x,T)dx\\
&-\int\limits_{\Tt^d} m^0(x)u(x,0)dx-\int\limits_{\Tt^d}m(x,T)u^T(x)dx\\
=&\int\limits_{\Omega_T} u(m_t-\epsi \Delta m)+m H_0(x,\nabla u)- F(x,m)dxdt+\int\limits_{\Tt^d}m(x,0)u(x,0)dx\\
&-\int\limits_{\Tt^d} m^0(x)u(x,0)dx-\int\limits_{\Tt^d}m(x,T)u^T(x)dx.
\end{split}
\end{equation}
Furthermore, $\Psi_2$ in \eqref{eq:Psi2} has the form
\begin{equation}\label{eq:Psi2_a=0_Hseparated}
\begin{split}
\Psi_2(m,u)=&\int\limits_{\Omega_T} m(-u_t-\epsi \Delta u)+m H_0\left(x,\nabla u\right)-mf(x,m)dxdt+\int\limits_{\Tt^d}m(x,T)u(x,T)dx\\
&-\int\limits_{\Tt^d} m^0(x)u(x,0)dx-\int\limits_{\Tt^d}m(x,T)u^T(x)dx\\
=&\int\limits_{\Omega_T} u(m_t-\epsi \Delta m)+ m H_0\left(x,\nabla u\right)-mf(x,m) dxdt+\int\limits_{\Tt^d}m(x,0)u(x,0)dx\\
&-\int\limits_{\Tt^d} m^0(x)u(x,0)dx-\int\limits_{\Tt^d}m(x,T)u^T(x)dx.
\end{split}
\end{equation}

Proposition \ref{prp:Nashnonzero} asserts that an optimal $u^*$ is a minimizer of $u\mapsto\Psi_2(m^*,u)$. Therefore, if we modify $\Psi_2(m,u)$ by adding a functional that depends only on $m$ the minimization problem in $u$ will not change. From \eqref{eq:Psi1_a=0_Hseparated} and \eqref{eq:Psi2_a=0_Hseparated} we observe that for $H$ as in \eqref{eq:H_separable} the functional $\Psi_2$ differs from $\Psi_1$ only by an $m$-dependent term $\int -m f(x,m) + F(x,m)$. Thus, in this case, Proposition \ref{prp:Nashnonzero} is valid with $\Psi_2$ replaced by $\Psi_1$.
\begin{corollary}\label{crl:Nash_a=0_Hseparated}
Suppose that \eqref{hyp:mon_convexity_a=0_Hseparated} holds. Furthermore, let $\Psi_1$ be given by \eqref{eq:Psi1_a=0_Hseparated}. Then, for any $m^*,u^* \in C^2(\Omega_T)$ such that $m^*>0$ one has that
\begin{equation}\label{eq:Psi1ineq_a=0_Hseparated}
\Psi_1(m^*,u^*) \geq \Psi_1(m,u^*),~ \mbox{for all}~m\in C^2(\Omega_T),
\end{equation}
and
\begin{equation}\label{eq:Psi2ineq_a=0_Hseparated}
\Psi_1(m^*,u^*) \leq \Psi_1(m^*,u),~ \mbox{for all}~u\in C^2(\Omega_T),
\end{equation}
if and only if $(m^*,u^*)$ is a classical solution of \eqref{eq:main_a=0_Hseparated}.
\end{corollary}
This previous corollary asserts that for $H$ as in \eqref{eq:H_separable} the smooth solutions of \eqref{eq:main} are Nash equilibria of the infinite-dimensional zero-sum game
\begin{equation}\label{eq:gameseparable}
\begin{split}
\mbox{Player 1}&\qquad\qquad \sup\limits_{m\in C^2(\Omega_T),~m>0} \Psi_1(m,u)\\
&\\
\mbox{Player 2}&\qquad\qquad \inf\limits_{u\in C^2(\Omega_T)} \Psi_1(m,u).
\end{split}
\end{equation}

Furthermore, for $H$ of the form \eqref{eq:H_separable} we have that $\hat{\Psi}_1,~\tilde{\Psi}_1$ from \eqref{eq:Psi1stat} and \eqref{eq:Psi12statLagrange} take the form
\begin{equation}\label{eq:Psi1stat_a=0_Hseparated}
\begin{split}
\hat{\Psi}_1(m,u)=&\int\limits_{\Tt^d} - \epsi m  \Delta u+m H_0(x,\nabla u)-F(x,m)dx,\\
\tilde{\Psi}_1(m,u,\Hh)=&\int\limits_{\Tt^d}  -\epsi m \Delta u+m H_0(x,\nabla u)-F(x,m)+\Hh(1-m)dx,
\end{split}
\end{equation}
for $m,u \in C^2(\Tt^d),~m>0,~\Hh\in \Rr$. Then, the following corollaries of Propositions \ref{prp:Nashstat} and \ref{prp:NashstatLagrange} hold.
\begin{corollary}\label{crl:Nashstat_a=0_Hseparated}
Suppose that \eqref{hyp:mon_convexity_a=0_Hseparated} holds. Furthermore, let $\hat{\Psi}_1$ be given by \eqref{eq:Psi1stat_a=0_Hseparated}. Then, for any $m^*,u^* \in C^{2}(\Tt^d)$ such that $m^*>0$ one has that
\begin{equation}\label{eq:Psi1ineqstat_a=0_Hseparated}
\hat{\Psi}_1(m^*,u^*) \geq \hat{\Psi}_1(m,u^*),~ \mbox{for all}~m\in C^{2}(\Tt^d),~m>0,~\int\limits_{\Tt^d}m=1.
\end{equation}
and
\begin{equation}\label{eq:Psi2ineqstat_a=0_Hseparated}
\hat{\Psi}_1(m^*,u^*) \leq \hat{\Psi}_1(m^*,u),~ \mbox{for all}~u\in C^{2}(\Tt^d),
\end{equation}
if and only if $(m^*,u^*)$ is a classical solution of \eqref{eq:mainstat_a=0_Hseparated} for some $\Hh^* \in \Rr$.

Moreover, for both cases above one has that
\begin{equation}\label{eq:Hstarstatseparable}
\begin{split}
\Hh^*=&\hat{\Psi}_2(m^*,u^*)=\hat{\Psi}_1(m^*,u^*)-\int\limits_{\Tt^d} m^* f(x,m^*)- F(x,m^*)dx\\
=&\hat{\Psi}_1(m^*,u^*)-\int\limits_{\Tt^d} F^*(x,f(x,m^*))dx,
\end{split}
\end{equation}
where $F^*(x,w)=\sup\limits_{z\in \Rr_{+}} w z - F(x,z),~w\in \Rr$, is the convex conjugate of $F$. 
\end{corollary}
The differential game corresponding to this previous corollary is
\begin{equation}\label{eq:gameseparablestat}
\begin{split}
\mbox{Player 1}&\qquad\qquad \sup\left\{\hat{\Psi}_1(m,u):~m \in C^2(\Tt^d),~m>0,~\int\limits_{\Tt^d}m=1\right\}\\
&\\
\mbox{Player 2}&\qquad\qquad \inf\left\{\hat{\Psi}_1(m,u):~u \in C^2(\Tt^d)\right\}.
\end{split}
\end{equation}

Next, we incorporate $\Hh$ in the game as in Proposition \ref{prp:NashstatLagrange}.
\begin{corollary}\label{crl:Nashstat_a=0_HseparatedLagrange}
Suppose that \eqref{hyp:mon_convexity_a=0_Hseparated} holds. Furthermore, let $\tilde{\Psi}_1$ be given by \eqref{eq:Psi1stat_a=0_Hseparated}. Then, for any $m^*,u^* \in C^{2}(\Tt^d),~\Hh^*\in \Rr$ such that $m^*>0$ one has that
\begin{equation}\label{eq:Psi1ineqstat_a=0_HseparatedLagrange}
\tilde{\Psi}_1(m^*,u^*,\Hh^*) \geq \tilde{\Psi}_1(m,u^*,\Hh^*),~ \mbox{for all}~m\in C^{2}(\Tt^d),~m>0.
\end{equation}
and
\begin{equation}\label{eq:Psi2ineqstat_a=0_HseparatedLagrange}
\tilde{\Psi}_1(m^*,u^*,\Hh^*) \leq \tilde{\Psi}_1(m^*,u,\Hh),~ \mbox{for all}~(u,\Hh)\in C^{2}(\Tt^d)\times \Rr,
\end{equation}
if and only if $(m^*,u^*,\Hh^*)$ is a classical solution of \eqref{eq:mainstat_a=0_Hseparated}.
\end{corollary}

The corresponding differential game is
\begin{equation}\label{eq:gameseparableLagrange}
\begin{split}
\mbox{Player 1}&\qquad\qquad \sup\limits_{m\in C^2(\Omega_T),~m>0} \tilde{\Psi}_1(m,u,\Hh)\\
&\\
\mbox{Player 2}&\qquad\qquad \inf\limits_{u\in C^2(\Omega_T),~\Hh\in \Rr} \tilde{\Psi}_1(m,u,\Hh).
\end{split}
\end{equation}

\begin{remark}
Note that the special structure of \eqref{eq:main} when $H$ is separable stems from the Hamiltonian nature of the system. Indeed, \eqref{eq:main} has the form
\[
\begin{cases}
u_t=\frac{\delta \hat{\Psi}_1}{\delta m}(m,u),\\
m_t=-\frac{\delta \hat{\Psi}_1}{\delta u}(m,u),\\
\end{cases}
\]
where $\hat{\Psi}_1$ is given by \eqref{eq:Psi1stat_a=0_Hseparated}. Note that if $(u, m)$ is a solution to \eqref{eq:main}, then $\hat{\Psi}_1(m,u)$ is a quantity that does not vary in time. It is well known (and easily observed) that solutions of Hamiltonian systems are associated to the corresponding functional
\[
\int\limits_0^T \frac{1}{2}\left(\int\limits_{\Tt^d} -u_tm + m_t u dx \right)+\hat{\Psi}_1( m, u) dt,
\]
that indeed coincides with $\Psi_1$ once it is equipped with initial-final data $m^0, u^T$ (see \eqref{eq:Psi1_a=0_Hseparated}). We will come back to the Hamiltonian nature of \eqref{eq:main} in Section \ref{sec:previous} where we discuss connections among \eqref{eq:gameseparable} and optimal-control formulations of \eqref{eq:main} from \cite{LL07}.
\end{remark}

\subsection{First-order stationary problems with congestion and a power-like Hamiltonian}

In this section, we consider the first-order ($\epsi=0$) version of \eqref{eq:mainstationary} with a Hamiltonian \eqref{eq:H_congestion1}, where $(Q,\alpha,\gamma)\in \Rr^{d+1}\times \Rr_{++}$ are given parameters, and $f\in C^1(\Tt^d\times \Rr_{++})$. More precisely, we consider the system
\begin{equation}\label{eq:mainstationary_cong}
\begin{cases}
\frac{|\nabla u+Q|^\gamma}{\gamma m^{\alpha}}=f(x,m)+\bar{H},\\
- \mathrm{div}  \left(m^{1-\alpha} |\nabla u+Q|^{\gamma-2}(\nabla u+Q)\right)=0,\\
m> 0,\ \int\limits_{\Tt^d} m(x)dx=1.
\end{cases}
\end{equation}
Furthermore, \eqref{hyp:mon_convexity} is equivalent to the following assumption.
\begin{hyp}
One has that
\begin{equation}\label{hyp:mon_conv_congestion}
\begin{split}
\alpha\geq0,~\gamma\geq 1\quad\mbox{and}\quad \partial_m f(x,m) \geq 0,\quad (x,m) \in \Tt^d \times \Rr_{++}.
\end{split}
\end{equation}
\end{hyp}

Next, we have that
\begin{equation*}
F_H(x,p,m)= \frac{m^{1-\alpha}|p+Q|^\gamma}{(1-\alpha)\gamma}-F(x,m),
\end{equation*}
where $F(x,m)=\int\limits^{m} f(x,z)dz$. Therefore, $\hat{\Psi}_1,\hat{\Psi}_2$ in \eqref{eq:Psi1stat}, \eqref{eq:Psi2stat} are given by
\begin{equation}\label{eq:hatPsi12_congestion}
\begin{split}
\hat{\Psi}_1(m,u)=&\int\limits_{\Tt^d} \frac{m^{1-\alpha}|\nabla u+Q|^\gamma}{(1-\alpha)\gamma}-F(x,m)dx,\\
\hat{\Psi}_2(m,u)=&\int\limits_{\Tt^d} \frac{m^{1-\alpha}|\nabla u+Q|^\gamma}{\gamma}-m f(x,m)dx,\quad m,u\in C^2(\Tt^d),~m>0.
\end{split}
\end{equation}
Furthermore, we observe that
\begin{equation}\label{eq:Psi1Psi2relation_congestion}
\hat{\Psi}_1(m,u)=\frac{1}{1-\alpha}\hat{\Psi}_2(m,u)+\int\limits_{\Tt^d} \frac{1}{1-\alpha}mf(x,m)-F(x,m)dx.
\end{equation}
Hence, for $\alpha<1$ minimizations in $u$ of $\hat{\Psi}_1$ and $\hat{\Psi}_2$ are equivalent, and Proposition \ref{prp:Nashstat} is valid with $\hat{\Psi}_2$ replaces by $\hat{\Psi}_1$.
\begin{corollary}\label{crl:Nashstat_congestion}
Suppose that \eqref{hyp:mon_conv_congestion} holds and $\alpha<1$. Furthermore, let $\hat{\Psi}_1$ be given by \eqref{eq:hatPsi12_congestion}. Then, for any $m^*,u^* \in C^{2}(\Tt^d)$ such that $m^*>0$ and $\int_{\Tt^d}m^*=1$ one has that
\begin{equation}\label{eq:Psi1ineqstat_congestion}
\hat{\Psi}_1(m^*,u^*) \geq \hat{\Psi}_1(m,u^*),~ \mbox{for all}~m\in C^{2}(\Tt^d),~m>0,~\int\limits_{\Tt^d}m=1.
\end{equation}
and
\begin{equation}\label{eq:Psi2ineqstat_congestion}
\hat{\Psi}_1(m^*,u^*) \leq \hat{\Psi}_1(m^*,u),~ \mbox{for all}~u\in C^{2}(\Tt^d),
\end{equation}
if and only if $(m^*,u^*)$ is a classical solution of \eqref{eq:mainstationary_cong} for some $\Hh^* \in \Rr$.

Moreover, for both cases above one has that
\begin{equation}\label{eq:Hstarcongestion}
\begin{split}
\Hh^*=&\hat{\Psi}_2(m^*,u^*)\\
=&(1-\alpha)\hat{\Psi}_1(m^*,u^*)-\int\limits_{\Tt^d} m^* f(x,m^*)-(1-\alpha)F(x,m^*)dx.
\end{split}
\end{equation}
\end{corollary}

Furthermore, $\tilde{\Psi}_1,\tilde{\Psi}_2$ in \eqref{eq:Psi12statLagrange} are given by
\begin{equation}\label{eq:hatPsi12_congestion_Lagrange}
\begin{split}
\tilde{\Psi}_1(m,u,\Hh)=&\int\limits_{\Tt^d} \frac{m^{1-\alpha}|p+Q|^\gamma}{(1-\alpha)\gamma}-F(x,m)+\Hh(1-m)dx,\\
\tilde{\Psi}_2(m,u,\Hh)=&\int\limits_{\Tt^d} \frac{m^{1-\alpha}|p+Q|^\gamma}{\gamma}-m f(x,m)+\Hh(1-m)dx,
\end{split}
\end{equation}
for $m,u\in C^2(\Tt^d),~m>0$. Similarly, for $\alpha<1$ minimizations in $u$ of $\tilde{\Psi}_1$ and $\tilde{\Psi}_2$ are equivalent, and an analog of Proposition \ref{prp:NashstatLagrange} is valid.
\begin{corollary}\label{crl:NashstatLagrange_congestion}
Suppose that \eqref{hyp:mon_conv_congestion} holds and $\alpha<1$. Furthermore, let $\tilde{\Psi}_1$ be given by \eqref{eq:hatPsi12_congestion_Lagrange}. Then, for any $m^*,u^* \in C^{2}(\Tt^d),~\Hh^*\in \Rr$ such that $m^*>0$ one has that
\begin{equation}\label{eq:Psi1ineqstat_congestionLagrange}
\tilde{\Psi}_1(m^*,u^*,\Hh^*) \geq \tilde{\Psi}_1(m,u^*,\Hh^*),~ \mbox{for all}~m\in C^{2}(\Tt^d),~m>0.
\end{equation}
and
\begin{equation}\label{eq:Psi2ineqstat_congestionLagrange}
\tilde{\Psi}_1(m^*,u^*,\Hh^*) \leq \tilde{\Psi}_1(m^*,u,\Hh),~ \mbox{for all}~(u,\Hh)\in C^{2}(\Tt^d)\times \Rr,
\end{equation}
if and only if $(m^*,u^*,\Hh^*)$ is a classical solution of \eqref{eq:mainstationary_cong}.
\end{corollary}

To the best of our knowledge, zero-sum game formulations of \eqref{eq:mainstationary_cong} for the parameters range $0\leq \alpha<1< \gamma$ in Corollaries \ref{crl:Nashstat_congestion}, \ref{crl:NashstatLagrange_congestion} are new in the literature.

\begin{remark}\label{rmk:BBformcong_a<1}
Using the zero-sum game formulation of \eqref{eq:mainstationary_cong} in Corollary \ref{crl:Nashstat_congestion} one can formally derive a convex optimization formulation of \eqref{eq:mainstationary_cong} in the spirit of \cite{benbren'00}. Indeed, one has that
\begin{equation*}
\begin{split}
&\inf\limits_{u} \sup\limits_m \hat{\Psi}_1(m,u)\\
=& \inf\limits_{u} \sup\limits_m \int\limits_{\Tt^d} \frac{m^{1-\alpha}|\nabla u+Q|^\gamma}{(1-\alpha)\gamma}-F(x,m)dx\\
=& \inf\limits_{u} \sup\limits_m \sup\limits_r \int\limits_{\Tt^d} \frac{m^{1-\alpha}}{(1-\alpha)}\left(r\cdot(\nabla u+Q)-\frac{|r|^{\gamma'}}{\gamma'}\right)-F(x,m)dx\\
=& \sup\limits_{m,r}  \inf\limits_{u} \int\limits_{\Tt^d} -\div \left(\frac{m^{1-\alpha}}{(1-\alpha)}r\right) u+\frac{m^{1-\alpha}}{(1-\alpha)} \left(r\cdot Q-\frac{|r|^{\gamma'}}{\gamma'}\right)-F(x,m)dx,
\end{split}
\end{equation*}
where $\gamma'=\frac{\gamma}{\gamma-1}$. Since the last expression in the previous formula is linear in $u$ we have that $\inf\limits_{u}$ is finite only when
\begin{equation}\label{eq:brenben_r_constraint}
\div(m^{1-\alpha} r)=0.
\end{equation}
Therefore, we obtain that
\begin{equation*}
\begin{split}
&\inf\limits_{u} \sup\limits_m \hat{\Psi}_1(m,u) \\
=& \sup\limits_{m,r}  \left\{\int\limits_{\Tt^d} \frac{m^{1-\alpha}}{(1-\alpha)} \left(r\cdot Q-\frac{|r|^{\gamma'}}{\gamma'}\right)-F(x,m)dx~\mbox{s.t.}~\eqref{eq:brenben_r_constraint}~\mbox{holds}\right\}
\end{split}.
\end{equation*}
Next, we denote by $w=m^{1-\alpha}r$, and \eqref{eq:brenben_r_constraint} becomes
\begin{equation}\label{eq:brenben_w_constraint}
\div (w)=0.
\end{equation}
Furthermore, we have that
\begin{equation*}
\begin{split}
&\sup\limits_{m,r}  \left\{\int\limits_{\Tt^d} \frac{m^{1-\alpha}}{(1-\alpha)} \left(r\cdot Q-\frac{|r|^{\gamma'}}{\gamma'}\right)-F(x,m)dx~\mbox{s.t.}~\eqref{eq:brenben_r_constraint}~\mbox{holds}\right\}\\
=&\sup\limits_{m,w} \left\{ \int\limits_{\Tt^d} \frac{1}{(1-\alpha)} w\cdot Q-\frac{|w|^{\gamma'}}{(1-\alpha)\gamma'm^{(\gamma'-1)(1-\alpha)}}-F(x,m)dx~\mbox{s.t.}~\eqref{eq:brenben_w_constraint}~\mbox{holds} \right\}\\
=&-\inf\limits_{m,w} \left\{ \int\limits_{\Tt^d} -\frac{1}{(1-\alpha)} w\cdot Q+\frac{|w|^{\gamma'}}{(1-\alpha)\gamma'm^{(\gamma'-1)(1-\alpha)}}+F(x,m)dx~\mbox{s.t.}~\eqref{eq:brenben_w_constraint}~\mbox{holds} \right\}.
\end{split}
\end{equation*}
For $m \in C^1(\Tt^d),m>0$ and $w\in C^1(\Tt^d;\Rr^d)$ denote by
\begin{equation}\label{eq:brenben_Phi}
\Phi(m,w)=\int\limits_{\Tt^d} -\frac{1}{(1-\alpha)} w\cdot Q+\frac{|w|^{\gamma'}}{(1-\alpha)\gamma'm^{(\gamma'-1)(1-\alpha)}}+F(x,m)dx.
\end{equation}
Therefore, \eqref{eq:mainstationary_cong} can be seen as optimality conditions for the variational problem
\begin{equation}\label{eq:brenben_cong_a<1}
\inf\limits_{m,w} \left\{ \Phi(m,w)~\mbox{s.t.}~m>0, \int\limits_{\Tt^d}m=1~\mbox{and}~\eqref{eq:brenben_w_constraint}~\mbox{holds} \right\},
\end{equation}
where $u$ is the adjoint variable corresponding to \eqref{eq:brenben_w_constraint}. The transformation from $(m,u)$ to $(m,w)$ and vice versa is given by
\begin{equation}\label{eq:u->w}
w=m^{1-\alpha}|\nabla u+Q|^{\gamma-2}(\nabla u+Q),
\end{equation}
and
\begin{equation}\label{eq:w->u}
\nabla u+Q = m^{(\alpha -1)(\gamma'-1)} |w|^{\gamma'-2} w.
\end{equation}
The key property of \eqref{eq:brenben_cong_a<1} is that \eqref{eq:brenben_w_constraint} is a linear constraint, and $(m,w)\mapsto \Phi(m,w)$ is a convex functional when $0\leq \alpha <1<\gamma$. Indeed, for $\alpha, \gamma$ in this range we have that the function $(a,b) \mapsto \frac{|b|^{\gamma'}}{a^{(\gamma'-1)(1-\alpha)}},~a>0$ is convex. Furthermore, $(m,w) \mapsto F(x,m)$ is convex by \eqref{hyp:mon_conv_congestion}.
\end{remark}

For $\alpha > 1$ we observe from \eqref{eq:Psi1Psi2relation_congestion} that the minimization of $u\mapsto\hat{\Psi}_2(m,u)$ corresponds to the maximization of $u\mapsto \hat{\Psi}_1(m,u)$. Therefore, we have the following versions of Corollaries \ref{crl:Nashstat_congestion}, \ref{crl:NashstatLagrange_congestion}. 
\begin{corollary}\label{crl:Nashstat_congestion_a>1}
Suppose that \eqref{hyp:mon_conv_congestion} holds and $\alpha>1$. Furthermore, let $\hat{\Psi}_1$ be given by \eqref{eq:hatPsi12_congestion}. Then, for any $m^*,u^* \in C^{2}(\Tt^d)$ such that $m^*>0$ and $\int_{\Tt^d}m^*=1$ one has that
\begin{equation}\label{eq:Psi1ineqstat_congestion_a>1}
\hat{\Psi}_1(m^*,u^*) \geq \hat{\Psi}_1(m,u^*),~ \mbox{for all}~m\in C^{2}(\Tt^d),~m>0,~\int\limits_{\Tt^d}m=1.
\end{equation}
and
\begin{equation}\label{eq:Psi2ineqstat_congestion_a>1}
\hat{\Psi}_1(m^*,u^*) \geq \hat{\Psi}_1(m^*,u),~ \mbox{for all}~u\in C^{2}(\Tt^d),
\end{equation}
if and only if $(m^*,u^*)$ is a classical solution of \eqref{eq:mainstationary_cong} for some $\Hh^* \in \Rr$.

Moreover, for both cases above, \eqref{eq:Hstarcongestion} holds. 
\end{corollary}

\begin{corollary}\label{crl:NashstatLagrange_congestion_a>1}
Suppose that \eqref{hyp:mon_conv_congestion} holds and $\alpha>1$. Furthermore, let $\tilde{\Psi}_1$ be given by \eqref{eq:hatPsi12_congestion_Lagrange}. Then, for any $m^*,u^* \in C^{2}(\Tt^d),~\Hh^*\in \Rr$ such that $m^*>0$ one has that
\begin{equation}\label{eq:Psi1ineqstat_congestionLagrange_a>1}
\tilde{\Psi}_1(m^*,u^*,\Hh^*) \geq \tilde{\Psi}_1(m,u^*,\Hh^*),~ \mbox{for all}~m\in C^{2}(\Tt^d),~m>0.
\end{equation}
and
\begin{equation}\label{eq:Psi2ineqstat_congestionLagrange_a>1}
\tilde{\Psi}_1(m^*,u^*,\Hh^*) \geq \tilde{\Psi}_1(m^*,u,\Hh),~ \mbox{for all}~(u,\Hh)\in C^{2}(\Tt^d)\times \Rr,
\end{equation}
if and only if $(m^*,u^*,\Hh^*)$ is a classical solution of \eqref{eq:mainstationary_cong}.
\end{corollary}
\begin{remark}\label{rmk:potGame_a>1}
Corollaries \ref{crl:Nashstat_congestion_a>1}, \ref{crl:NashstatLagrange_congestion_a>1} assert that \eqref{eq:mainstationary} corresponds to Nash equilibria of
\begin{equation}\label{eq:potgame}
\begin{split}
\mbox{Player 1}&\qquad\qquad \sup\left\{\hat{\Psi}_1(m,u):~m \in C^2(\Tt^d),~m>0,~\int\limits_{\Tt^d}m=1\right\}\\
&\\
\mbox{Player 2}&\qquad\qquad \sup\left\{\hat{\Psi}_1(m,u):~u \in C^2(\Tt^d)\right\},
\end{split}
\end{equation}
and
\begin{equation}\label{eq:potgameLagrange}
\begin{split}
\mbox{Player 1}&\qquad\qquad \sup\limits_{m\in C^2(\Omega_T),~m>0} \tilde{\Psi}_1(m,u,\Hh)\\
&\\
\mbox{Player 2}&\qquad\qquad \sup\limits_{u\in C^2(\Omega_T),~\Hh\in \Rr} \tilde{\Psi}_1(m,u,\Hh),
\end{split}
\end{equation}
Unlike \eqref{eq:gameseparablestat} and \eqref{eq:gameseparableLagrange} that are zero-sum games \eqref{eq:potgame} and \eqref{eq:potgameLagrange} are potential games with potentials $\hat{\Psi}_1$ and $\tilde{\Psi}_1$, respectively.
\end{remark}

\section{Previously known variational principles}\label{sec:previous}

In this section, we discuss how infinite-dimensional differential game formulations in Propositions \ref{prp:Nashnonzero}, \ref{prp:Nashstat}, \ref{prp:NashstatLagrange} are related to the variational principles in the MFG literature.

\subsection{Infinite-dimensional optimal control formulation}

Throughout this section we assume that $H$ is given by \eqref{eq:H_separable}, where $H_0 \in C^2(\Tt^d \times \Rr^d),~f\in C^1(\Tt^d\times \Rr_{++})$, and \eqref{hyp:mon_convexity_a=0_Hseparated} holds. As before, $F(x,m)=\int\limits^m f(x,z)dz,~m>0$, and
\begin{equation*}
F^*(x,w) =\sup\limits_{z\in \Rr_{+}} w z - F(x,z),~w\in \Rr.
\end{equation*}
Furthermore, denote by
\begin{equation*}
L_0(x,q)=\sup\limits_{p \in \Rr^d} q\cdot p - H_0(x,p),
\end{equation*}
the convex conjugate of $p\mapsto H_0(x,p)$. If $p\mapsto H_0(x,p)$ is strictly convex and coercive uniformly in $x$ then $L_0\in C^2(\Tt^d,\Rr^d)$, and $q \mapsto L_0(x,q)$ is strictly convex and coercive uniformly in $x$. Moreover,
\begin{equation*}
H_0(x,p)+L_0(x,q) \geq p \cdot q,\quad \mbox{for all}\quad (x,p,q)\in \Tt^d \times \Rr^d \times \Rr^d
\end{equation*}
with an equality if and only if
\begin{equation*}
q=\nabla_p H_0(x,p) \quad \mbox{or}\quad p=\nabla_q L_0(x,q).
\end{equation*}

In \cite{LL07}, Lasry and Lions observed that \eqref{eq:main} is equivalent to two infinite-dimensional optimal control problems. For the first problem, consider the cost functional
\begin{equation}\label{eq:Bcost}
\mathcal{B}(r)= \int\limits_{\Omega_T} L_0(x,-r(x,t)) m(x,t)+F(x,m(x,t))dxdt+\int\limits_{\Tt^d} u^T(x) m(x,T) dx,
\end{equation}
where $r:\Omega_T \to \Rr^d$ is the control, $m$ is the state, and the dynamics is according to a Fokker-Planck equation
\begin{equation}\label{eq:FPdynamics}
\begin{cases}
m_t(x,t)-\epsi \Delta m(x,t) + \div \big( m(x,t) r(x,t) \big)=0,\\
m(x,0)=m^0(x),~ (x,t) \in \Omega_T.
\end{cases}
\end{equation}
Thus, the first optimal control problem is
\begin{equation}\label{eq:FPcontrol}
\inf \{\mathcal{B}(r)~\mbox{s.t.}~\eqref{eq:FPdynamics}~\mbox{holds}\}
\end{equation}
For the second problem consider the functional
\begin{equation}\label{eq:Acost}
\mathcal{A}(s)=\int\limits_{\Omega_T} F^*(x,s(x,t))dxdt-\int\limits_{\Tt^d} u(x,0) m^0(x)dx,
\end{equation}
where $s:\Omega_T\to \Rr$ is the control, $u$ is the state, and the dynamics is given by
\begin{equation}\label{eq:HJdynamics}
\begin{cases}
-u_t(x,t)-\epsi \Delta u(x,t) +H_0\big(x,\nabla u(x,t)\big)=s(x,t),\\
u(x,T)=u^T(x),~(x,t)\in \Omega_T.
\end{cases}
\end{equation}
Hence, the second optimal control problem is
\begin{equation}\label{eq:HJcontrol}
\inf \{\mathcal{A}(s)~\mbox{s.t.}~\eqref{eq:HJdynamics}~\mbox{holds}\}
\end{equation}
In \cite{LL07}, the authors observe that \eqref{eq:FPcontrol}, \eqref{eq:HJcontrol} are dual optimization problems (in the sense of Fenchel-Rockafellar) that yield \eqref{eq:main} as optimality conditions. Additionally, under convexity assumptions for $p\mapsto H_0(x,p)$ and $m \mapsto F(x,m)$ the cost functionals $s\mapsto \mathcal{A}(s)$ and $r\mapsto\mathcal{B}(r)$ are convex. 

Here, we observe that the differential-game formulation \eqref{eq:gameseparable} is the Hamiltonian viewpoint for the optimal control problems \eqref{eq:FPcontrol} and \eqref{eq:HJcontrol}. For that, we briefly recall the Hamiltonian formalism for finite-dimensional optimal control problems and apply it to \eqref{eq:FPcontrol}, \eqref{eq:HJcontrol} to arrive at \eqref{eq:gameseparable}.

Consider a finite-dimensional optimal control problem
\begin{equation}\label{eq:finitecontrol}
\begin{split}
&\inf\limits_{c}\int\limits_a^b l(\gamma(\tau),c(\tau)) d \tau+\phi(\gamma(b))\\
&\dot{\gamma}(\tau)=k(\gamma(\tau),c(\tau)),~\tau \in (a,b)\\
&\gamma(a)=x_0
\end{split}
\end{equation}
In \eqref{eq:finitecontrol} $\gamma$ is the state variable and $c$ is the control. To obtain the Hamiltonian formulation of \eqref{eq:finitecontrol} one introduces an adjoint variable, $p$, and transforms \eqref{eq:finitecontrol} into an equivalent problem
\begin{equation*}
\begin{split}
&\inf\limits_{\gamma,c} \sup\limits_p \int\limits_a^b l(\gamma(\tau),c(\tau))-p(\tau)\big(\dot{\gamma}(\tau)-k(\gamma(\tau),c(\tau))\big) d \tau +\phi(\gamma(b))\\
&\gamma(a)=\gamma_0
\end{split}
\end{equation*}
Next, we proceed by formally interchanging the $\inf \sup$ order and eliminating the control (Pontryagin maximum principle):
\begin{equation*}
\begin{split}
&\inf\limits_{\gamma,c} \sup\limits_p \int\limits_a^b l(\gamma(\tau),c(\tau))-p(\tau)\big(\dot{\gamma}(\tau)-k(\gamma(\tau),c(\tau))\big) d \tau +\phi(\gamma(b))\\
=&\sup\limits_p \inf\limits_{\gamma,c}  \int\limits_a^b l(\gamma(\tau),c(\tau))-p(\tau)\big(\dot{\gamma}(\tau)-k(\gamma(\tau),c(\tau))\big) d \tau +\phi(\gamma(b))\\
=&\sup\limits_p \inf\limits_{\gamma} \inf\limits_{c}  \int\limits_a^b l(\gamma(\tau),c(\tau))-p(\tau)\big(\dot{\gamma}(\tau)-k(\gamma(\tau),c(\tau))\big) d \tau +\phi(\gamma(b))\\
=&\sup\limits_p \inf\limits_{\gamma}   \int\limits_a^b -p(\tau) \dot{\gamma}(\tau)-h(\gamma(\tau),p(\tau)) d \tau +\phi(\gamma(b)),
\end{split}
\end{equation*}
where
\begin{equation*}
h(\gamma,p)= \sup\limits_{c} -p\cdot k(\gamma,c) - L(\gamma,c).
\end{equation*}
Therefore, \eqref{eq:finitecontrol} is formally equivalent to the problem
\begin{equation}\label{eq:Hamform}
\begin{split}
&\sup\limits_p \inf\limits_{\gamma}   \int\limits_a^b -p(\tau) \dot{\gamma}(\tau)-h(\gamma(\tau),p(\tau)) d \tau +\phi(\gamma(b))\\
&\gamma(a)=\gamma_0.
\end{split}
\end{equation}
If we calculate the variational derivatives of \eqref{eq:Hamform} with respect to $p$ and $\gamma$ we arrive at the system
\begin{equation}\label{eq:Hamsystem}
\begin{cases}
\dot{\gamma}(\tau)=-\nabla_p h(\gamma(\tau),p(\tau)),\quad \gamma(a)=\gamma_0,\\
\dot{p}(\tau)=\nabla_\gamma h(\gamma(\tau),p(\tau)),\quad p(b)=\nabla \phi(\gamma(b)).
\end{cases}
\end{equation}
This previous system is the Hamiltonian formulation of \eqref{eq:finitecontrol}. Now, let us formally apply the procedure above to \eqref{eq:FPcontrol} and \eqref{eq:HJcontrol}.

For \eqref{eq:FPcontrol} one has that
\begin{equation*}
\begin{split}
&\inf\limits_{r,~m(x,0)=m^0(x)} \int\limits_{\Omega_T} L_0(x,-r) m+F(x,m)dxdt+\int\limits_{\Tt^d} u^T(x) m(x,T) dx\\
=&\inf\limits_{m,r} \sup\limits_u \int\limits_{\Omega_T} L_0(x,-r) m+F(x,m)-u\big(m_t-\epsi \Delta m+ \div(m r) \big)dxdt \\
&+\int\limits_{\Tt^d} u^T(x) m(x,T) dx-\int\limits_{\Tt^d}u(x,0)m(x,0)dx+\int\limits_{\Tt^d}m^0(x)u(x,0)dx\\
=&\sup\limits_u \inf\limits_m \inf\limits_r \int\limits_{\Omega_T} L_0(x,-r) m+ m r \nabla u+F(x,m)+m(u_t+\epsi \Delta u)  dxdt \\
&+\int\limits_{\Tt^d} u^T(x) m(x,T) dx-\int\limits_{\Tt^d}u(x,T)m(x,T)dx+\int\limits_{\Tt^d}m^0(x)u(x,0)dx\\
=&\sup\limits_u \inf\limits_m \int\limits_{\Omega_T} -m H_0(x,\nabla u)+F(x,m)+m(u_t+\epsi \Delta u)  dxdt \\
&+\int\limits_{\Tt^d} u^T(x) m(x,T) dx-\int\limits_{\Tt^d}u(x,T)m(x,T)dx+\int\limits_{\Tt^d}m^0(x)u(x,0)dx\\
=&\sup\limits_u \inf\limits_m -\Psi_1(m,u)=- \inf\limits_u \sup\limits_m \Psi_1 (m,u),
\end{split}
\end{equation*}
where $\Psi_1$ is given by \eqref{eq:Psi1_a=0_Hseparated}. Therefore, we arrive at \eqref{eq:gameseparable} where Player 1 makes the first move.

For \eqref{eq:HJcontrol} one has that
\begin{equation*}
\begin{split}
&\inf\limits_{s,~u(x,T)=u^T(x)} \int\limits_{\Omega_T} F^*(x,s)dxdt-\int\limits_{\Tt^d} u(x,0) m^0(x)dx\\
=&\inf\limits_{s,u} \sup \limits_m \int\limits_{\Omega_T} F^*(x,s)+m(-u_t -\epsi \Delta u + H_0\big(x,\nabla u\big)-s)dxdt+\int\limits_{\Tt^d}u(x,T)m(x,T)dx\\
&-\int\limits_{\Tt^d} u^T(x) m(x,T) dx-\int\limits_{\Tt^d} u(x,0) m^0(x)dx\\
&= \sup \limits_m \inf\limits_{u} \inf\limits_{s} \int\limits_{\Omega_T} F^*(x,s)-m s +m\big(-u_t -\epsi \Delta u + H_0(x,\nabla u)\big)dxdt+\int\limits_{\Tt^d}u(x,T)m(x,T)dx\\
&-\int\limits_{\Tt^d} u^T(x) m(x,T) dx-\int\limits_{\Tt^d} u(x,0) m^0(x)dx\\
=& \sup \limits_m \inf\limits_{u} \int\limits_{\Omega_T} F(x,m) +m\big(-u_t -\epsi \Delta u + H_0(x,\nabla u)\big)dxdt+\int\limits_{\Tt^d}u(x,T)m(x,T)dx\\
&-\int\limits_{\Tt^d} u^T(x) m(x,T) dx-\int\limits_{\Tt^d} u(x,0) m^0(x)dx\\
=& \sup \limits_m \inf\limits_{u} \Psi_1(m,u),
\end{split}
\end{equation*}
where $\Psi_1$ is again given by \eqref{eq:Psi1_a=0_Hseparated}. Thus, we obtain \eqref{eq:gameseparable} where Player 2 makes the first move.

In \cite{car'15} and subsequent papers \cite{graber'14,CardaGraber,grames'18,CGPT,MeSil,carmesa'16} the authors considered a modification of \eqref{eq:FPcontrol} in the spirit of \cite{benbren'00}. More precisely, they chose as a control $w =m r$ instead of $r$ and considered the optimization problem
\begin{equation}\label{eq:BrenBencontrol}
\begin{split}
\inf\limits_{m,w} \int\limits_{\Omega_T} L_0\left(x,-\frac{w(x,t)}{m(x,t)}\right) m(x,t)+F(x,m(x,t))dxdt+\int\limits_{\Tt^d} u^T(x) m(x,T) dx,\\
m_t(x,t)-\epsi \Delta m(x,t) + \div (w(x,t))=0,~ m(x,0)=m^0(x),~(x,t) \in \Omega_T.
\end{split}
\end{equation}
The advantage of \eqref{eq:BrenBencontrol} over \eqref{eq:FPcontrol} is that the former is a convex optimization problem in $(m,w)$ with a linear constraint whereas \eqref{eq:FPcontrol} is not jointly convex in $(m,r)$ and \eqref{eq:FPdynamics} is not a jointly linear constraint in $(m,r)$.

\subsection{First-order stationary MFG with congestion}

In \cite{evafegonuvos'18}, the authors observed that under assumptions \eqref{hyp:mon_conv_congestion} and $1<\alpha\leq \gamma$ \eqref{eq:mainstationary_cong} admits a variational formulation
\begin{equation}\label{eq:minJ}
\min\limits_{m,u}J(m,u)=\min\limits_{m,u}\int\limits_{\Tt^d} \frac{m^{\alpha-1}|\nabla u+Q|^\gamma}{(1-\alpha)\gamma}+F(x,m)dx,
\end{equation}
where $F(x,m)=\int\limits^{m} f(x,z)dz$.

We observe that $J(m,u)=-\hat{\Psi}_1(m,u)$, where $\hat{\Psi}_1$ is given by \eqref{eq:hatPsi12_congestion}. Therefore, variational formulation \eqref{eq:minJ} follows from Corollary \ref{crl:Nashstat_congestion_a>1} and the fact that $(m,u) \mapsto \hat{\Psi}_1(m,u),~m>0$ is a concave functional for $1<\alpha \leq \gamma$. 

Additionally the authors observed in \cite{evafegonuvos'18} that $(m,u)\mapsto J(m,u),~m>0$ is not convex when $0<\alpha<1<\gamma$, and hence \eqref{eq:minJ} is not valid. Accordingly, they applied a suitable transformation to $(m,u)$ and obtained a new pair $(m,v)$ that solves a related system of the form \eqref{eq:mainstationary_cong} with parameters $1< \tilde{\alpha}<\tilde{\gamma}$. Consequently, they obtained a convex optimization problem similar to \eqref{eq:minJ} for $(m,v)$ that allowed to find $(m,u)$ by first finding $(m,v)$ and then applying an inverse transformation.

Unfortunately, the technique in \cite{evafegonuvos'18} is valid only in the two-dimensional setting because of the special structure of divergence free vector fields in two-dimensions.

Here, we observe that variational formulation \eqref{eq:brenben_cong_a<1} is the generalization of the one in \cite{evafegonuvos'18} to the higher-dimensional setting. Indeed, if $d=2$, then \eqref{eq:brenben_w_constraint} yields that
\begin{equation*}
w=(\nabla v)^\perp+R^\perp,
\end{equation*}
for some $v:\Tt^2 \mapsto \Rr$ and $R \in \Rr^2$, and $^\perp$ is the rotation by $\pi/2$. Therefore, we have that
\begin{equation*}
\begin{split}
\Phi(m,w)=&\int\limits_{\Tt^d} -\frac{1}{(1-\alpha)} w\cdot Q+\frac{|w|^{\gamma'}}{(1-\alpha)\gamma'm^{(\gamma'-1)(1-\alpha)}}+F(x,m)dx\\
=&\int\limits_{\Tt^d} -\frac{1}{(1-\alpha)} \left((\nabla v)^\perp+R^\perp\right)\cdot Q+\frac{|(\nabla v)^\perp+R^\perp|^{\gamma'}}{(1-\alpha)\gamma'm^{(\gamma'-1)(1-\alpha)}}+F(x,m)dx\\
=&\int\limits_{\Tt^d} -\frac{1}{(1-\alpha)} \left(\nabla v+R\right)\cdot Q^\perp+\frac{|\nabla v+R|^{\gamma'}}{(1-\alpha)\gamma'm^{(\gamma'-1)(1-\alpha)}}+F(x,m)dx\\
=&-\frac{R\cdot Q^\perp}{(1-\alpha)}+\int\limits_{\Tt^d} \frac{|\nabla v+R|^{\gamma'}}{(1-\alpha)\gamma'm^{(\gamma'-1)(1-\alpha)}}+F(x,m)dx.
\end{split}
\end{equation*}
Therefore, \eqref{eq:brenben_cong_a<1} is equivalent to
\begin{equation}\label{eq:brenben_(m,v)}
\inf\limits_{m,v} \left\{ \int\limits_{\Tt^d} \frac{|\nabla v+R|^{\gamma'}}{(1-\alpha)\gamma'm^{(\gamma'-1)(1-\alpha)}}+F(x,m)dx.~\mbox{s.t.}~m>0, \int\limits_{\Tt^d}m=1 \right\}.
\end{equation}
This previous formulation is precisely the one obtained in \cite{evafegonuvos'18} for $(m,v)$. Finally, note that \eqref{eq:brenben_(m,v)} is almost identical to \eqref{eq:minJ}.

\section{Existence of non-trivial periodic solutions to variational MFG}\label{sec:bifu}

In this section, we present some results on the existence of periodic in time solutions that are based on the aforementioned variational structure. We will assume that $H(x, p, m) = \frac{1}{2}|p|^2 - f(m)$, where $f$ is a smooth decreasing function. For simplicity, $\epsi = 1$.

We look for a solution to \eqref{eq:main} such that $m(\cdot, t)$ is defined for all $t \in (-\infty, +\infty)$ and $m(\cdot, t + T) = m(\cdot, t)$ for some $T > 0$ and for all $t$. Periodicity in time of $u$ is more subtle and cannot be expected in general: for Hamilton-Jacobi equations with periodic data one usually looks for quasi-periodic solutions; that is, solutions $\phi$ such that $\phi(\cdot, t + T) = \phi(\cdot, t) + \overline H T,~\forall t$ for some $\overline H \in \Rr$ and period $T$ (see, e. g., \cite{BS, EG}). For such $\phi$, the function $u(\cdot,t)=\phi(\cdot,t)-\Hh t$ is $T$-periodic. Therefore, we search for a triple $(u, \overline H, m)$, where $u, m$ are $T$-periodic in the $t$ variable and solve
\begin{equation*}
\begin{cases}
-u_t - \Delta u + \frac{1}{2}|\nabla u|^2 =f(m) + \bar{H},\\
m_t -\Delta m - \mathrm{div}  \left(\nabla u \, m\right)=0,\\
m > 0,\ \int\limits_{\Tt^d} m(x,0)dx=1.
\end{cases}
\end{equation*}
Note, that for smooth solution of this system the $t\mapsto \int_{\Tt^d}m(x,t)dx$ is a conserved quantity. Therefore, for smooth solutions this previous system is equivalent to
\begin{equation}\label{eq:periodic}
\begin{cases}
-u_t - \Delta u + \frac{1}{2}|\nabla u|^2 =f(m) + \bar{H},\\
m_t -\Delta m - \mathrm{div}  \left(\nabla u \, m\right)=0,\\
m > 0,\ \int\limits_{\Omega_T} m(x,t)dxdt= T.
\end{cases}
\end{equation}
Our first observation is that concerning periodic solutions \eqref{eq:periodic} can be cast as a stationary (ergodic) MFG system rather than a time-dependent one. The only difference with \eqref{eq:mainstationary} is that the suitable Hamiltonian is linear in the gradient variable $u_t$ and the Laplacian is degenerate in the $t$ direction. Indeed, consider the Hamiltonian
\[
\tilde{H}(x,t,p,q,m)=-q+\frac{|p|^2}{2},
\]
and a degenerate diffusion
\[A u(x,t)=\Delta_x u(x,t). 
\]
Then, \eqref{eq:periodic} can be written as
\begin{equation}\label{eq:periodicstat}
\begin{cases}
 - A u + \tilde{H}(x,t,\nabla_x u,u_t) =f(m)+ \bar{H},\\
-A^*m  - \mathrm{div}_{x,t}  \left(\nabla_{p,q} \tilde{H}(x,t,\nabla_x u,u_t)  m\right)=0,\\
m > 0,\ \int\limits_{\Omega_T} m(x,t)dxdt= T.
\end{cases}
\end{equation}
Furthermore, as we observed before, this previous system can be fitted into a variational framework. Indeed, \eqref{eq:periodicstat} has precisely the same analytic structure as \eqref{eq:mainstationary}. Therefore, from Corollary \ref{crl:Nashstat_a=0_HseparatedLagrange} we have that \eqref{eq:periodicstat} can be cast as
\begin{equation}\label{eq:periodiccritical}
\begin{cases}
\frac{\delta \tilde{\Psi}_1}{\delta m}(u,m,\Hh)=0,\\
\frac{\delta \tilde{\Psi}_1}{\delta u}(u,m,\Hh)=0,\\
\frac{\delta \tilde{\Psi}_1}{\delta \Hh}(u,m,\Hh)=0,
\end{cases}
\end{equation}
where
\begin{equation}\label{eq:tildePsi1periodic}
\tilde{\Psi}_1(u,m,\Hh)=\int\limits_{\Omega_T}  - m~ A u+m \tilde{H}(x,t,\nabla_x u,u_t)-F(m)+\Hh(T-m)dx,
\end{equation}
and $F(m)=\int^m f(z)dz$. The fact that \eqref{eq:periodic} can be written as \eqref{eq:periodiccritical} is crucial in our analysis. To take advantage of bifurcation methods, we perform a change of variables as follows:
\[
\begin{cases}
U(x,t) := u(x,T t) + t T f(1),  \\
M(x,t) := m(x,T t) -1.
\end{cases}
\]
Then, $U$ and $M$ are functions over $Q = \Tt^{d+1}$ and solve 
\begin{equation}\label{eq:periodic2}
\begin{cases}
-\frac{1}{T} U_t - \Delta U + \frac{1}{2}|\nabla U|^2 =f(M + 1) - f(1) + \bar{H},\\
\frac{1}{T} M_t -\Delta M - \Delta U - \mathrm{div}  \left(\nabla U \, M\right)=0,\\
M > 0,\ \int\limits_{Q} M dx dt = \int\limits_{Q} U dx dt=0.
\end{cases}
\end{equation}
We require $\int_{Q} U dx dt=0$ since $U$ in the Hamilton-Jacobi equation is defined up to an addition of constants. Note that $T$ itself should be regarded as an unknown of the problem.

As \eqref{eq:periodicstat}, the system \eqref{eq:periodic2} can also be cast as an equation for critical points of a suitable functional, $g$, that we obtain below by expressing $\tilde{\Psi}_1$ in \eqref{eq:tildePsi1periodic} in $(U,M,\Hh)$ variables. First, \eqref{eq:periodic2} can be restated in terms of the zero-locus of a suitable functional $G$. Let $\alpha \in (0, 1)$, $X$ be the Banach space
\[
X := \left\{ (U, M) \in C^{4 + \alpha, 2 + \alpha/2}(Q)^2 : \int_{Q} U dx dt=0 \right\} \times \Rr
\]
and $G : X \times \Rr \to C^{2 + \alpha, 1 + \alpha/2}(Q)^2 \times \Rr$ by
\begin{multline*}
G(U, M, \overline H, T) = \left(\frac{1}{T} M_t - \Delta M - \Delta U - \diverg(\nabla U M),\right.\\ \left. -\frac{1}{T} U_t - \Delta U + \frac{1}{2}|\nabla U|^2 - f(M+1) + f(1) + \lambda, \int_Q M dx dt \right).
\end{multline*}
Note that $ \int_{Q} G_1(U, M, \overline H, T) dx dt = 0$ for all $(U, M, \overline H, T)$. Moreover,
\[
G(0, 0, 0, T) = 0 \qquad \text{for all $T> 0$,}
\]
that is, $G$ has a trivial solution for all $T > 0$. The change of variables is indeed designed for the trivial solution to be identically zero, to apply bifurcation theory.

In the rest of this section, we will aim at proving the following
\begin{theorem}\label{thm:bifu} Suppose that $-8\pi^2<f'(1)<-4\pi^2$ and let
\[
\overline T = \frac{1}{\sqrt{-4\pi^2 - f'(1)}}
\]
Then, $(0, 0, 0, \overline T)$ is a bifurcation point for the equation $G(U, M, \overline H, T) = 0$. In other words, there exists a sequence of non-trivial solutions $(U_n, M_n, \overline H_n, T_n) \subset X \times \Rr$ to \eqref{eq:periodic2} such that $(U_n, M_n, \overline H_n, T_n) \to (0, 0, 0, \overline T)$ as $n \to \infty$.
\end{theorem}

Going back to the original unknowns, Theorem \ref{thm:bifu} states that there exists a sequence of non-trivial solutions $(u_n, \overline H_n, m_n)$ to \eqref{eq:periodic} that is $T_n$-periodic, that is, couples $(u_n, m_n)$ solving \eqref{eq:main} such that $m_n$ is $T_n$-periodic.

Let us start by some comments on the functional setting. Since $C^{4 + \alpha, 2 + \alpha/2}(Q) \subset L^2(Q)$, we can consider a scalar product on $X$ defined by $\langle (U_1, M_1, \overline H_1), (U_2, M_2, \overline H_2) \rangle =  \int_Q M_1 M_2 + U_1 U_2 dx dt + \overline H_1 \overline H_2$. Passing from $(u,m,\Hh)$ variables to $(U,M,\Hh)$ in \eqref{eq:tildePsi1periodic} we obtain a functional $g : X \times \Rr \to \Rr$ given by
\begin{multline*}
g(U, M, \overline H, T) = \\ \int_Q -\frac{U_t M}{T} + \nabla U \cdot \nabla M + \frac{1}{2}|\nabla U|^2 (M+1) -F(M+1) + f(1) M +\overline H M \, dx dt.
\end{multline*}
It is quite standard to prove that $g$ is differentiable (when $T > 0$). Then, $G$ is a {\it potential} operator in the following sense:
\[
D_{(U, M, 	\overline H)} g(U, M, 	\overline H, T)[v, \mu, \ell] = \langle G(U, M, 	\overline H, T) , (v, \mu, \ell) \rangle \qquad \forall (v, \mu, \ell),
\]
the equality following by integration by parts.

A crucial role will be played by the properties of the linearized problem $D G = 0$. We compute the Fr\'echet derivative of $G$,
\begin{multline*}
D_{(U, M, 	\overline H)} G(U, M, \overline H,T)[v, \mu, \ell] =  \left(\frac{1}{T} \mu_t - \Delta \mu - \Delta v - \diverg(\nabla U \mu + M \nabla v ),\right.\\ \left. -\frac{1}{T} v_t - \Delta v + \nabla U \cdot  \nabla v - f'(M+1)\mu + \ell, \int_Q \mu dx dt \right).
\end{multline*}
Evaluating it at the trivial solution $(0, 0, 0)$ gives the linear operator $A(T)$
\begin{multline*}
A(T)[v, \mu, \ell] := D_{(U, M, 	\overline H)} G(0, 0, 0, T)[v, \mu, \ell] \\ =  \left(\frac{1}{T} \mu_t - \Delta \mu - \Delta v, -\frac{1}{T} v_t - \Delta v  - f'(1)\mu + \ell, \int_Q \mu dx dt \right).
\end{multline*}

\begin{lemma}\label{linearized} Zero is an (isolated) eigenvalue of $A(\overline T)$ with multiplicity $4d$. Moreover,  $A(\overline T)$ is a Fredholm operator of index zero\footnote{That is, $A(\overline T)$ has kernel and cokernel of same finite dimension.}.
\end{lemma}

To prove Lemma \ref{linearized}, we perform an analysis on the Fourier coefficients of $v, \mu$: let us denote by  $(\lambda_k)_{k \ge 0}$ the non-decreasing sequence  of eigenvalues of $-\Delta$ (on $\Tt^N$), with corresponding eigenvectors $\psi_k \in C^{\infty}(\Tt^N)$; let $(\psi_k)_{k \ge 0}$ be renormalized such that it constitutes an orthonormal basis of $L^2(\Tt^N)$. Note that the first eigenvalue $\lambda_0$ is zero, with associated constant eigenfunction $\psi_0 \equiv 1$. Non-zero eigenvalues can be expressed in the form $\lambda_{k_1k_2\cdots k_N}=4\pi^2\sum_{i=1}^{N}k_i^2$, where $k_i$ are nonnegative integers. Therefore, $\lambda_1= \lambda_2 = \ldots = \lambda_{2d} = 4 \pi^2 < \lambda_{2d + 1}$. 

We may represent $v, \mu$ as follows:
\[
\mu(x,t) = \sum_{k \ge 0} \mu_{k}(t) \psi_k(x), \qquad v(x,t) = \sum_{k \ge 0} v_{k}(t) \psi_k(x),
\]
where $v_k, \mu_k$ is a family of (time) periodic functions of class $C^{2+\alpha/2}(\Tt)$. In general, we will refer to the Fourier coefficients by using the subscript $k$, i.e. $f_k(\cdot) := \int_{\Tt^N} f(x,\cdot) \psi_k(x) dx$. 

\begin{proof}[Proof of Lemma \ref{linearized}] Consider the equation $A(\overline T)[v, \mu, \ell] = 0$ in terms of Fourier coefficients. The 0-th Fourier coefficient of $A(\overline T)[v,\mu,l]$ yields the following system.
\begin{equation*}
\begin{cases}
\mu_0'=0,\\
-v_0'- \overline Tf'(1)\mu_0+T \ell=0,\\
\int_{\Tt} \mu_0(t)dt=0.
\end{cases}
\end{equation*}
By periodicity, we obtain $\mu_0 \equiv 0$. Furthermore, $\ell=0$, and $v_0\equiv 0$ since $\int_{\Tt}v_0(t)dt = \int_Q v(x,t)dxdt=0$. On the other hand, for all $k \ge 1$ we have
\[
\begin{cases}
\mu_k' + \overline T\lambda_k \mu_k + \overline T\lambda_k v_k = 0, \\
- v_k'  + \overline T\lambda_k v_k - \overline T f'(1) \mu_k = 0
\end{cases}
\]
This system of ODEs is equivalent to
\begin{equation}\label{ode}
\begin{cases}
\mu_k'' = \overline T^2 \lambda_k (\lambda_k + f'(1)) \mu_k,\\
\overline T\lambda_k v_k = - \mu_k' - \overline T\lambda_k \mu_k.
\end{cases}
\end{equation}

By the standing assumptions, for all $k > 2d$ we have $\lambda_k + f'(1) \ge \lambda_{2d + 1} + f'(1) = 8\pi^2 + f'(1) > 0$, so $\mu_k \equiv 0$ (and in turn $v_k \equiv 0$). Finally, since $\overline T^2 \lambda_k (\lambda_k + f'(1)) = 4 \pi^2$ when $k = 1, \ldots, 2d$, solutions $(\mu_k, v_k)$ to \eqref{ode} consist of a two-dimensional vector space spanned by
\[
\begin{split}
&\left(\cos(2\pi t), \frac{\sqrt{-4\pi^2 - f'(1)}}{2\pi} \sin(2\pi t) - \cos(2\pi t)\right), \\
&\left(\sin(2\pi t), -\frac{\sqrt{-4\pi^2 - f'(1)}}{2\pi} \cos(2\pi t) - \sin(2\pi t)\right).
\end{split}
\]
Therefore, we have that $A(\overline T)$ has a non-trivial kernel of dimension $4d$.

To prove that cokernel $A(\overline T)$ has the same dimension, we proceed analogously. For any given triple $(\omega, \nu, a)$ (note that $\omega$ can be assumed such that $\int_Q \omega = 0$), the equation $A(\overline T)[v, \mu, \ell] = (\omega, \nu, a)$ reads as follows. For $k = 0$,
\begin{equation*}
\begin{cases}
\mu_0'=\omega_0,\\
-v_0'- \overline Tf'(1)\mu_0+ \overline T \ell= \nu_0,\\
\int_{\Tt} \mu_0(t)dt=a.
\end{cases}
\end{equation*}
The first equation determines $\mu_0$, while the other two ones determine first $\ell$ by integration on $\Tt$ and then $v_0$, recalling that $\int_{\Tt} v_0  = 0$. For all $k > 2d$,
\[
\begin{cases}
-\mu_k'' + \overline T^2 \lambda_k (\lambda_k + f'(1)) \mu_k = \omega_k' - \overline T \lambda_k(\omega_k + \overline T \ell - \nu_k),\\
\overline T\lambda_k v_k = \omega_k - \mu_k' - \overline T\lambda_k \mu_k.
\end{cases}
\]
has uniquely determined solutions $\mu_k, v_k$. For $k = 1, \ldots, 2d$, we have to argue via Fredholm alternative: $\mu_1$ (and therefore $v_1$) if and only if $\omega_k' - \overline T \lambda_k(\omega_k - \nu_k)$ is $L^2(\Tt)$-orthogonal to $\cos(2\pi \cdot)$ and $\sin(2\pi \cdot)$. Hence, a solution $[v, \mu, \ell]$ exists if and only if $\omega' - \overline T \lambda_k(\omega - \nu)$ is $L^2(Q)$-orthogonal to $\cos(2\pi t) \psi_k(x)$ and $\sin(2\pi t)\psi_k(x)$ for all $k = 1, \ldots, 2d$.
\end{proof}

To prove Theorem \ref{thm:bifu}, we need to define the {\it crossing number} of $A(T)$ through $0$ at $T = \overline T$. First, for any $T > 0$, denote by $\sigma^<(T)$ the sum of the multiplicities of all perturbed eigenvalues of $A(T)$ near $0$ on the negative real axis. For some $\delta > 0$, $\sigma^<(T)$ is constant for all $T \in (\overline T-\delta, \overline T)$ and for all $T \in (\overline T, \overline T + \delta)$. The crossing number is then defined by \[\chi(A(T), \overline T) = \sigma^<(\overline T- \epsilon) - \sigma^<(\overline T +  \epsilon), \qquad  0 < \epsilon < \delta.\]

\begin{proof}[Proof of Theorem \ref{thm:bifu}] We will apply the following result that can be found in \cite{Kiel88} or \cite[Theorem II.7.3]{Kiel}.
\begin{theorem}\label{bifu} Assume that $A(\overline T)$ is a Fredholm operator of index zero having an isolated eigenvalue zero. If $\chi(A(T), \overline T)$ is nonzero, then $(0, 0, 0, \overline T)$ is a bifurcation point for the equation $G(u, m, \lambda, T) = 0$. \end{theorem}

The assumptions of this theorem are satisfied in view of Lemma \ref{linearized}. Let us compute the eigenvalues $\sigma = \sigma(T)$ of $A(T)$ close to zero for $T$ close to $\overline T$, with their multiplicities. Consider the equation $A(T)[v, \mu, \ell] = \sigma(v, \mu, \ell)$ in terms of the Fourier expansion. For the 0-th order term we get the following system.
\begin{equation*}
\begin{cases}
\mu_0'=\sigma v_0,\\
-v_0'-Tf'(1)\mu_0+T \ell=\sigma \mu_0,\\
\int_{\Tt}\mu_0(t)dt=\sigma \ell.
\end{cases}
\end{equation*}
Integrating the second equation over $\Tt$ we get
\[-Tf'(1)\sigma \ell +T \ell = \sigma^2 \ell,\quad \mbox{or}\quad \ell(T-\sigma T f'(1)-\sigma^2)=0.
\]
Therefore, if $\sigma$ is small we get that $\ell=0$. Thus, we arrive at
\[
\begin{cases}
\mu_0'=\sigma v_0,\\
-v_0'-Tf'(1)\mu_0=\sigma \mu_0,\\
\int_{\Tt}\mu_0(t)dt=0.
\end{cases}
\]
But then, we get that
\[\begin{cases}
\mu_0''+\sigma(\sigma+Tf'(1))\mu_0=0,\\
\int_{\Tt}\mu_0=0.
\end{cases}
\]
Therefore, if $\mu_0 \neq 0$ we get that $\sigma(\sigma+Tf'(1))=4 \pi^2 n^2$ for some $n \geq 1$ which is impossible when $\sigma$ is small. Hence, for $\sigma$ small enough we have that $\mu_0\equiv 0$. Latter, in turn, yields $v_0\equiv 0$.

For all $k \ge 1$
\[
\begin{cases}
\mu_k' + T\lambda_k \mu_k + T\lambda_k v_k = \sigma v_k, \\
- v_k'  + T\lambda_k v_k - T f'(1) \mu_k = \sigma \mu_k
\end{cases}
\]
that is equivalent to
\begin{equation}\label{mueq}
\mu_k'' = [T^2 \lambda_k(\lambda_k +  f'(1)) + \sigma T( \lambda_k -  f'(1)) - \sigma^2] \mu_k.
\end{equation}
As before, we have that $\mu_k\equiv 0$ for all $k\geq 2d+1$ for $\sigma$ small enough, since $T^2 \lambda_k(\lambda_k +  f'(1))$ is positive and bounded away from zero for all $T$ in a neighbourhood of $\overline T$. Therefore, we need to consider the equation for $1\leq k \leq 2d$. 

Note that $\lambda_k = \lambda_1$. Let \[h(T, \sigma) := [T^2 \lambda_1(\lambda_1 +  f'(1)) + \sigma T( \lambda_1 -  f'(1)) - \sigma^2] + 4 \pi^2.\]
We have $h(\overline T, 0)  = 0$ and $\partial_\sigma h(\overline T, 0)  = \overline T( \lambda_1 -  f'(1)) > 0$, so by the Implicit Function Theorem, there exists a $C^1$ function $\sigma(\cdot)$ defined in a neighborhood of $\overline T$ such that $h(T, \sigma(T)) = 0$. One can also verify that $\sigma'(\overline T) > 0$. Indeed, we have that
\[\sigma'(\overline T)=-\frac{\partial_T h(\overline T,0)}{\partial_\sigma h(\overline T,0)}=\frac{2 \lambda_1(-\lambda_1-f'(1))}{\lambda_1-f'(1)}>0.
\]

Whenever $h(T, \sigma(T)) = 0$, \eqref{mueq} has non-trivial solutions of the form
\[
\mu_k(t) = A \cos(2\pi t) + B \sin(2\pi t), \qquad A, B \in \Rr.
\]
In other words, $\sigma(T)$ is (the only) eigenvalue in a neighbourhood of zero of $A(T)$, with multiplicity $2\cdot 2 d=4d$. The eigenfunctions are (we just write the $\mu$ entry for brevity)
\[\{\cos (2\pi t) \cos (2 \pi  x_i),\ \cos (2\pi t) \sin (2 \pi  x_i),\ \sin (2\pi t) \cos (2 \pi  x_i),\ \sin (2\pi t) \sin (2 \pi  x_i)\}_{i=1}^N.
\]

 Since $\sigma(T) < 0$ for $T < \overline T$ and $\sigma(T) > 0$ for $T > \overline T$ (in a neighbourhood of $\overline T$), we have a non-zero crossing number, in particular $\chi(A(T), \overline T) = 4d$. Hence Theorem \ref{bifu} applies.
 
 \end{proof}
 
 \begin{remark} Very little can be said about qualitative properties of $U_n, M_n$, but we can rule out the chance that they do not depend on time and that are just functions of the $x$-variable. In other words, our bifurcation result does not select non-trivial solutions that are stationary. We have to go through the proof of Theorem \ref{bifu}, that is based on the well-known Lyapunov-Schmidt reduction (see, e.g., \cite{Kiel}). Denote by $\Uu = (U, M, \overline H) \in X$ and by $P : X \to {\rm ker}(A(\overline T))$ the $L^2$-orthogonal projection. Then,
 \[
 G(\Uu, T) = 0 \Leftrightarrow
 \begin{cases}
 PG(\Vv + \Ww, T) = 0 & \Vv = P\Uu, \\
 (I-P)G(\Vv + \Ww, T) = 0 & \Ww = (I-P)\Uu.
 \end{cases}
 \]
 The Implicit Function Theorem assures the solvability of the second equation, namely
 \[
 \Ww = \psi(\Vv, T), \qquad \psi(0, T) = 0, \qquad \partial_\Vv \psi(0, \overline T) = 0.
 \]
 On the other hand, the finite dimensional {\it bifurcation equation}
 \begin{equation}\label{eq:be}
 PG(\Vv + \psi(\Vv, T), T) = 0, \qquad \Vv \in {\rm ker}(A(\overline T))
 \end{equation}
 is solved by means of Conley's index theory applied to the dynamical system $\dot{\Vv} = PG(\Vv + \psi(\Vv, T), T)$. Bifurcation is actually proven for the finite-dimensional problem, namely, there exists a sequence of non-trivial solutions $(\Vv_n, T) \to (0, \overline T)$ to \eqref{eq:be}. Recall that ${\rm ker}(A(\overline T))$ is spanned by functions of the form $\sin(2\pi t)\psi_k(x)$ and $\cos(2\pi t)\psi_k(x)$ in the $\mu$ entry (see Lemma \ref{linearized}). Therefore, $M_n$ in $\Uu_n = (U_n, M_n, \overline H_n)$ cannot be constant in the $t$-variable, otherwise $P \Uu_n = \Vv_n$ would be zero, contradicting the property that $ \Vv_n$ is non-trivial.
 
 Note that by the properties of $\psi$ we get the expansion
 \[
 \Uu_n = \Vv_n + o(|T_n- \overline T| + \|\Vv_n\|) \qquad \text{as $\Vv_n \to 0$}.
 \] 
 
 \end{remark}
 
\begin{remark} Note that $T^N = N(-4\pi^2 - f'(1))^{-1/2} = N \overline T$, for any integer $N \ge 2$, are other bifurcation times, namely there are sequences of solutions $(U^N_n, M^N_n, \overline H^N_n, T^N_n)$ clustering at $(0, 0, 0, T^N)$. Indeed, arguing as in Lemma \ref{linearized}, ${\rm ker}(A(T^N))$ is non-trivial. We observe that those solutions should not be qualitatively different from the ones of Theorem \ref{thm:bifu}, at least if one considers the original variables $u, m$.

Let us analyze the sequence $m_n$, and recall that $m_n(x, t) = 1 + M_n(x, t/T_n)$. By the previous remark, one should get that $M_n(x, t/T_n)$ is approximately an element of ${\rm ker}(A(\overline T))$, so it is a linear combination of $\sin(2\pi t/T_n)\psi_k(x)$ and $\cos(2\pi t/T_n)\psi_k(x)$, with $1/T_n \to 1/\overline T$. On the other hand, $m^N_n(x, t) = 1 + M^N_n(x, t/T^N_n)$, where $M^N_n(x, t/T^N_n)$ is approximately an element of ${\rm ker}(A(T^N))$. Such kernel is spanned by $\sin(2\pi N t)\psi_k(x)$ and $\cos(2\pi N t)\psi_k(x)$, therefore, $M^N_n(x, t/T^N_n)$ is approximately a linear combination of $\sin(2\pi N t/T^N_n)\psi_k(x)$ and $\cos(2\pi N t/T^N_n)\psi_k(x)$, and $N / T^N_n \to 1/\overline T$ independently on $N$.
 \end{remark}

\bibliographystyle{AIMS}

\providecommand{\href}[2]{#2}
\providecommand{\arxiv}[1]{\href{http://arxiv.org/abs/#1}{arXiv:#1}}
\providecommand{\url}[1]{\texttt{#1}}
\providecommand{\urlprefix}{URL }

\end{document}